\documentclass[11pt, final]{article}
\usepackage{a4}

\usepackage{amsmath}
\usepackage{amssymb}
\usepackage{amsxtra}
\usepackage{amsthm}
\usepackage{bbm}			
\usepackage{cite}

\usepackage[left=3.5cm,right=3.5cm,top=3cm,bottom=3cm]{geometry}

\usepackage[usenames,dvipsnames]{xcolor}
\usepackage[plainpages=false,pdfpagelabels,colorlinks=true,linkcolor=RoyalPurple,citecolor=Magenta]{hyperref} 

\usepackage{showkeys}
\usepackage[latin1]{inputenc}       
\usepackage[T1]{fontenc}        
\usepackage{lmodern}            
\usepackage{graphicx}
\usepackage{subfig}		
\usepackage{booktabs}		
\usepackage{multirow}
\usepackage{enumitem}		
\usepackage{listings} 	
\usepackage{mathtools}  
\usepackage{floatrow}		

\usepackage{colortbl}
\definecolor{dunkelgrau}{rgb}{0.8,0.8,0.8}
\definecolor{hellgrau}{rgb}{0.9,0.9,0.9}

\newcommand{\R}{\ensuremath{\mathbb{R}}}
\newcommand{\N}{\ensuremath{\mathbb{N}}}

\newcommand{\oR}{\ensuremath{\overline{\mathbb{R}}}}

\renewcommand{\>}{\right\rangle}
\newcommand{\<}{\left\langle}
\newcommand{\Prox}{\ensuremath{\text{Prox}}}
\newcommand{\id}{\ensuremath{\text{Id}}}
\renewcommand{\O}{\ensuremath{\mathcal{O}}}
\newcommand{\iso}{\ensuremath{\text{iso}}}
\newcommand{\aniso}{\ensuremath{\text{aniso}}}

\newcommand{\bx}{\ensuremath{\overline{x}}}

\newcommand{\bv}{\ensuremath{\overline{v}}}

\newcommand{\bp}{\ensuremath{\overline{p}}}

\renewcommand{\bv}{\ensuremath{\overline{v}}}

\newcommand{\fx}{\ensuremath{\boldsymbol{x}}}
\newcommand{\fv}{\ensuremath{\boldsymbol{v}}}
\newcommand{\fbv}{\ensuremath{\boldsymbol{\overline{v}}}}

\newcommand{\fbx}{\ensuremath{\boldsymbol{\overline{x}}}}
\newcommand{\fp}{\ensuremath{\boldsymbol{p}}}
\newcommand{\fq}{\ensuremath{\boldsymbol{q}}}
\newcommand{\fr}{\ensuremath{\boldsymbol{r}}}
\newcommand{\fy}{\ensuremath{\boldsymbol{y}}}
\newcommand{\fz}{\ensuremath{\boldsymbol{z}}}
\newcommand{\fL}{\ensuremath{\boldsymbol{L}}}
\newcommand{\ff}{\ensuremath{F}}
\newcommand{\fg}{\ensuremath{G}}
\newcommand{\fA}{\ensuremath{\boldsymbol{A}}}
\newcommand{\fB}{\ensuremath{\boldsymbol{B}}}
\newcommand{\fC}{\ensuremath{\boldsymbol{C}}}

\newcommand{\fG}{\ensuremath{\boldsymbol{\mathcal{G}}}}
\newcommand{\fH}{\ensuremath{\boldsymbol{\mathcal{H}}}}

\newcommand{\g}{\ensuremath{\mathcal{G}}}
\newcommand{\h}{\ensuremath{\mathcal{H}}}

\newcommand{\Y}{\ensuremath{\mathcal{Y}}}
\newcommand{\proj}{\ensuremath{\mathcal{P}}}

\renewcommand{\Box}{\ensuremath{\mbox{\small$\,\square\,$}}}

\theoremstyle{plain}
\newtheorem{theorem}{Theorem}[section]

\newtheorem{proposition}[theorem]{Proposition}

\theoremstyle{definition}
\newtheorem{remark}{Remark}[section]
\newtheorem{example}{Example}[section]

\newtheorem{problem}{Problem}[section]
\newtheorem{algorithm}{Algorithm}[section]

\DeclareMathOperator*\dom{dom}%
\DeclareMathOperator*\gra{gra}%
\DeclareMathOperator*\argmin{arg\,min}%
\DeclareMathOperator*\ran{ran}%

\numberwithin{equation}{section}  

\title{Convergence analysis for a primal-dual monotone + skew splitting algorithm with applications to total variation minimization}

\author{Radu Ioan Bo\c t
\thanks {Faculty of Mathematics, Chemnitz University of Technology, D-09107 Chemnitz, Germany, e-mail: radu.bot@mathematik.tu-chemnitz.de. Research partially supported by DFG (German Research Foundation), project BO 2516/4-1.}
\and Christopher Hendrich
\thanks{Faculty of Mathematics, Chemnitz University of Technology, D-09107 Chemnitz, Germany, e-mail: christopher.hendrich@mathematik.tu-chemnitz.de. Research supported by a Graduate Fellowship of the Free State Saxony, Germany.}
}
\date{\today}

\begin{document}
\maketitle

{\bf Abstract.} In this paper we investigate the convergence behavior of a primal-dual splitting method for solving monotone inclusions involving mixtures of composite, Lipschitzian and parallel sum type operators proposed by Combettes and Pesquet in \cite{ComPes12}. Firstly, in the particular case of convex minimization problems, we derive convergence rates for the sequence of objective function values by making use of conjugate duality techniques. Secondly, we propose for the general monotone inclusion problem two new schemes which accelerate the sequences of primal and/or dual iterates, provided strong monotonicity assumptions for some of the involved operators are fulfilled. Finally, we apply the theoretical achievements in the context of different types of image restoration problems solved via total variation regularization.



\section{Introduction and preliminaries}\label{sectionIntro}

The last few years have shown a rising interest in solving structured nondifferentiable convex optimization problems within the framework of the theory of conjugate functions. Applications in fields like signal and image processing, location theory and supervised machine learning motivate these efforts.

In this article we investigate and improve the convergence behavior of the primal-dual monotone + skew splitting method for solving monotone inclusions which was proposed by Combettes and Pesquet in \cite{ComPes12}, itself being an extension of the algorithmic scheme from \cite{BriCom11} obtained by allowing also Lipschitzian monotone operators and parallel sums in the problem formulation. In the mentioned works, by means of a product space approach, the problem is reduced to the one of finding the zeros of the sum of a Lipschitzian monotone operator with a maximally monotone operator. The latter is solved by using an
error-tolerant version of Tseng's algorithm which has forward-backward-forward characteristics and allows to access the monotone Lipschitzian operators via explicit forward steps, while set-valued maximally monotone operators are processed via their resolvents. A notable advantage of this method is given by both its highly parallelizable character, most of its steps could be executed independently, and by the fact that allows to process maximal monotone operators and linear bounded operators separately, whenever they occur in the form of precompositions in the problem formulation.

Before coming to the description of the problem formulation and of the algorithm from \cite{ComPes12}, we introduce some preliminary notions and results which are needed throughout the paper.

We are considering the real Hilbert spaces $\h$ and $\g_i$, $i=1,\ldots,m$, endowed with the \textit{inner product} $\left\langle \cdot ,\cdot \right\rangle$ and associated \textit{norm} $\left\| \cdot \right\| = \sqrt{\left\langle \cdot, \cdot \right\rangle}$, for which we use the same notation, respectively, as there is no risk of confusion. The symbols $\rightharpoonup$ and $\rightarrow$ denote weak and strong convergence, respectively. By $\R_{++}$ we denote the set of strictly positive real numbers, while the \textit{indicator function} of a set $C \subseteq \h$ is $\delta_C : \h \rightarrow  \oR := \R \cup \left\{ \pm \infty \right\}$, defined by $\delta_C(x) = 0$ for $x \in C$ and $\delta_C(x) = +\infty$, otherwise. For a function $f: \h \rightarrow \oR$ we denote by $\dom f := \left\{ x \in \h : f(x) < +\infty \right\}$ its \textit{effective domain} and call $f$ \textit{proper} if $\dom f \neq \varnothing$ and $f(x)>-\infty$ for all $x \in \h$. Let be
$$\Gamma(\h) := \{f: \h \rightarrow \overline \R: f \ \mbox{is proper, convex and lower semicontinuous}\}.$$
The \textit{conjugate function} of $f$ is $f^*:\h \rightarrow \oR$, $f^*(p)=\sup{\left\{ \left\langle p,x \right\rangle -f(x) : x\in\h \right\}}$ for all $p \in \h$ and, if $f \in \Gamma(\h)$, then $f^* \in \Gamma(\h)$, as well. The \textit{(convex) subdifferential} of $f: \h \rightarrow \oR$ at $x \in \h$ is the set $\partial f(x) = \{p \in \h : f(y) - f(x) \geq \left\langle p,y-x \right\rangle \ \forall y \in \h\}$, if $f(x) \in \R$, and is taken to be the empty set, otherwise. For a linear continuous operator $L_i: \h \rightarrow \g_i$, the operator $L_i^*: \g_i \rightarrow \h$, defined via $\< L_ix,y  \> = \< x,L_i^*y  \>$ for all $x \in \h$ and all $y \in \g_i$, denotes its \textit{adjoint operator}, for $i\in \{1,\ldots,m\}$.

Having two functions $f,\,g : \h \rightarrow \oR$, their \textit{infimal convolution} is defined by $f \Box g : \h \rightarrow \oR$, $(f \Box g) (x) = \inf_{y \in \h}\left\{ f(y) + g(x-y) \right\}$ for all $x \in \h$, being a convex function when $f$ and $g$ are convex.

Let $M:\h \rightarrow 2^{\h}$ be a set-valued operator. We denote by $\gra M = \{ (x,u) \in \h \times \h : u \in Mx\}$ its \textit{graph} and by $\ran M =\{u \in \h : \exists x \in \h,\ u\in Mx\}$ its \textit{range}. The \textit{inverse operator} of $M$ is defined as $M^{-1} :\h \rightarrow 2^{\h}, M^{-1}(u) = \{x \in H: u \in Mx\}$. The operator $M$ is called \textit{monotone} if $\< x-y,u-v \> \geq 0$ for all $(x,u),\,(y,v) \in \gra M$ and it is called \textit{maximally monotone} if there exists no monotone operator $M':\h \rightarrow 2^{\h}$ such that $\gra M'$ properly contains $\gra M$ . The operator $M$ is called $\rho$-strongly monotone, for $\rho \in \R_{++}$, if $M-\rho \id$ is monotone, i.\,e. $\< x-y,u-v \> \geq \rho \| x-y \|^2$  for all $(x,u),\,(y,v) \in \gra M$,
where $\id$ denotes the identity on $\h$. The operator $M : \h \rightarrow \h$ is called $\nu$-Lipschitzian for $\nu \in \R_{++}$ if it is single-valued and it fulfills $\|Mx-My\| \leq \nu \|x-y\|$ for all $x,y \in \h$.

The resolvent of a set-valued operator $M:\h \rightarrow 2^{\h}$ is $J_M:\h \rightarrow 2^{\h}, J_M = \left( \id + M \right)^{-1}$. When $M$ is maximally monotone, the resolvent is a single-valued, $1$-Lipschitzian and maximal monotone operator.  Moreover, when $f \in \Gamma(\h)$ and $\gamma \in \R_{++}$, $\partial (\gamma f)$ is maximally monotone (cf. \cite[Theorem 3.2.8]{Zalinescu02}) and it holds $J_{\gamma \partial f} = \left(\id + \gamma \partial f \right)^{-1} = \Prox_{\gamma f}$. Here, $\Prox_{\gamma f}(x)$ denotes the \textit{proximal point} of parameter $\gamma$ of $f$ at $x \in \h$ and it represents the unique optimal solution of the optimization problem
\begin{equation}\label{prox-def}
\inf_{y\in \h}\left \{f(y)+\frac{1}{2\gamma}\|y-x\|^2\right\}.
\end{equation}
For a nonempty, convex and closed set $C \subseteq \h$ and $\gamma \in \R_{++}$ we have $\Prox_{\gamma \delta_C} = \mathcal{P}_C$, where
$\mathcal{P}_C : \h \rightarrow C$, $\mathcal{P}_C(x) = \argmin_{z\in C}\left\| x-z \right\|$, denotes the \textit{projection operator} on $C$.

Finally, the \textit{parallel sum} of two set-valued operators $M_1, M_2: \h \rightarrow 2^{\h}$ is defined as
$$M_1 \Box M_2 : \h \rightarrow 2^{\h},  M_1 \Box M_2  = \left(M_1^{-1} + M_2^{-1}\right)^{-1}.$$
We can formulate now the monotone inclusion problem which we investigate in this paper (see \cite{ComPes12}).
\begin{problem}
\label{opt-problem-inclusion-full}
Consider the real Hilbert spaces $\h$ and $\g_i, i=1,...,m,$ $A:\h \rightarrow 2^{\h}$ a maximally monotone operator and $C:\h \rightarrow \h$ a monotone and $\mu$-Lipschitzian operator for some $\mu \in \R_{++}$. Furthermore, let $z \in \h$ and for every $i\in\{1,\ldots,m\}$, let $r_i \in \g_i$, let $B_i : \g_i \rightarrow 2^{\g_i}$ be maximally monotone operators, let $D_i:\g_i \rightarrow 2^{\g_i}$ be monotone operators such that $D_{i}^{-1}$ is $\nu_i$-Lipschitzian for some $\nu_i \in \R_{++}$, and let $L_i : \h \rightarrow \g_i$ be a nonzero linear continuous operator. The problem is to solve the primal inclusion
\begin{align}
	\label{opt-problem-primal-inclusion-full}
	\text{find }\bx \in \h \text{ such that } z \in A\bx + \sum_{i=1}^m L_i^* \left( (B_i\Box D_i)(L_i \bx-r_i) \right) + C\bx,
\end{align}
together with the dual inclusion
\begin{align}
	\label{opt-problem-dual-inclusion-full}
	\text{find }\bv_1 \in \g_1,\ldots,\bv_m \in \g_m \text{ such that }(\exists x\in\h)\left\{
	\begin{array}{l}
		z - \sum_{i=1}^m L_i^*\bv_i \in Ax + Cx \\
		\!\!\bv_i \!\in \!(B_i \Box D_i)(L_ix-r_i), \!i=1,\ldots,m.
	\end{array}
\right.
\end{align}
\end{problem}
Throughout this paper we denote by  $\fG:=\g_1 \times... \times \g_m$
the Hilbert space equipped with the inner product
$$\langle (p_1,\ldots,p_m), (q_1,\ldots,q_m) \rangle =  \sum_{i=1}^m \langle p_i,q_i \rangle \ \forall (p_1,\ldots,p_m) \ \forall (q_1,\ldots,q_m) \in \fG$$
and the associated norm $\|(p_1,\ldots,p_m)\| =  \sqrt{\sum_{i=1}^m \| p_i \|^2}$ for all $(p_1,\ldots,p_m) \in \fG$. We introduce also the nonzero linear continuous operator $\fL:\h \rightarrow \fG$, $\fL x = (L_1 x,\ldots,L_m x)$, its adjoint being $\fL^* : \fG \rightarrow \h$, $\fL^* \fv = \sum_{i=1}^m L_i^* v_i$.

We say that $(\bx, \bv_1,\ldots,\bv_m) \in \h \times \fG$ is a primal-dual solution to Problem \ref{opt-problem-inclusion-full}, if
\begin{equation}\label{operator-proof-conditions-full}
z - \sum_{i=1}^m L_i^*\bv_i \in A\bx + C\bx \ \mbox{and} \  \bv_i \in (B_i \Box D_i)(L_i\bx-r_i), \,i=1,\ldots,m.
\end{equation}
If $(\bx, \bv_1,\ldots,\bv_m) \in \h \times \fG$ is a primal-dual solution to Problem \ref{opt-problem-inclusion-full}, then $\bx$ is a solution to \eqref{opt-problem-primal-inclusion-full} and $(\bv_1,\ldots,\bv_m)$ is a solution to \eqref{opt-problem-dual-inclusion-full}. Notice also that
\begin{align*}
\bx \text{ solves }\eqref{opt-problem-primal-inclusion-full} & \Leftrightarrow z - \sum_{i=1}^m L_i^* (B_i\Box D_i)(L_i \bx-r_i)  \in A\bx + C\bx\Leftrightarrow\\
\exists\, \bv_1\in\g_1,\ldots,\bv_m\in\g_m \ & \mbox{such that} \ \left\{
		\begin{array}{l} z - \sum_{i=1}^m L_i^*\bv_i  \in A\bx + C\bx, \\ \bv_i \in  (B_i\Box D_i)(L_i \bx-r_i) , \ i=1,\ldots,m.  \end{array}\right.
\end{align*}
Thus, if $\bx$ is a solution to \eqref{opt-problem-primal-inclusion-full}, then there exists $(\bv_1,\ldots,\bv_m) \in \fG$ such that $(\bx, \bv_1,\ldots,\bv_m)$ is a primal-dual solution to Problem \ref{opt-problem-inclusion-full} and if $(\bv_1,\ldots,\bv_m)$ is a solution to \eqref{opt-problem-dual-inclusion-full}, then there exists $\bx \in \h$ such that $(\bx, \bv_1,\ldots,\bv_m)$ is a primal-dual solution to Problem \ref{opt-problem-inclusion-full}.

The next result provides the error-free variant of the primal-dual algorithm in \cite{ComPes12} and the corresponding convergence statements, as given in \cite[Theorem 3.1]{ComPes12}.
\begin{theorem}
	\label{theorem-inclusion-prelim}
	For Problem \ref{opt-problem-inclusion-full} suppose that
	$$ z \in \ran\left( A + \sum_{i=1}^m L_i^*\left( B_i \Box D_i \right)(L_i\cdot -r_i) + C \right).$$
	Let $x_0 \in \h$ and $(v_{1,0}, \ldots, v_{m,0}) \in \fG$, set
	$$ \beta = \max\{\mu,\nu_1\ldots,\nu_m\} + \sqrt{\sum_{i=1}^m \|L_i\|^2},$$
	choose $\varepsilon \in (0,\frac{1}{\beta+1})$ and $(\gamma_n)_{n\geq 0}$ a sequence in $\left[\varepsilon, \frac{1-\varepsilon}{\beta}\right]$ and set
	\begin{align}\label{A0}
	  \left(\forall n\geq 0\right) \   \left\lfloor \begin{array}{l}
		p_{1,n} = J_{\gamma_n A}\left(x_n - \gamma_n \left( Cx_n +\sum_{i=1}^mL_i^* v_{i,n} -z\right) \right) \\
		\text{For }i=1,\ldots,m  \\
				\ \left\lfloor \begin{array}{l}
					p_{2,i,n} = J_{\gamma_n B_i^{-1}}\left(v_{i,n} +\gamma_n (L_i x_n  -D_i^{-1}v_{i,n} -r_i)\right) \\
					v_{i,n+1} = \gamma_n L_i( p_{1,n} - x_n) + \gamma_n(D_i^{-1}v_{i,n} - D_i^{-1}p_{2,i,n})  + p_{2,i,n} \\
				\end{array} \right.\\
		x_{n+1} = \gamma_n \sum_{i=1}^m L_i^*(v_{i,n}-p_{2,i,n}) + \gamma_n(C x_n - C p_{1,n}) +p_{1,n} 	.
		\end{array}
		\right.
	\end{align}
	Then the following statements are true:	
	\begin{enumerate}[label={(\roman*)}]
	\setlength{\itemsep}{-2pt}
		\item  $\sum_{n\in\N}\|x_n-p_{1,n}\|^2 < +\infty$ and $\forall i \in \{1,\ldots,m\}$ $\sum_{n\in\N}\|v_{i,n}-p_{2,i,n}\|^2 < +\infty$.
		\item  There exists a primal-dual solution $(\bx,\bv_1,\ldots,\bv_m) \in \h \times \fG$ to Problem \ref{opt-problem-inclusion-full} such that the
following hold:
\begin{enumerate}[label={(\alph*)}]
			\setlength{\itemsep}{-2pt}
			\item $x_n \rightharpoonup \bx$ and $p_{1,n}\rightharpoonup \bx$.
			\item $(\forall i \in \{1,\ldots,m\})\ v_{i,n} \rightharpoonup \bv_i$ and $p_{2,i,n}\rightharpoonup \bv_i$.
\end{enumerate}
\end{enumerate}
\end{theorem}
In this paper we consider first Problem \ref{opt-problem-inclusion-full} in its particular formulation as a primal-dual pair of convex minimization problems, approach which relies on the fact that the subdifferential of a proper, convex and lower semicontinuous function is maximally monotone, and show that the convergence rate of the sequence of objective function values on the iterates generated by \eqref{A0} is of $\O(\frac{1}{n})$, where  $n \in \N$ is the number of passed iterations. Further, in Section \ref{sectionOperator},  we provide for the general monotone inclusion problem, as given in Problem \ref{opt-problem-inclusion-full}, two new acceleration schemes which generate under strong monotonicity assumptions sequences of primal and/or dual iterates converge with improved convergence properties. The feasibility of the proposed methods is explicitly shown in Section \ref{sectionApp} by means of numerical experiments in the context of solving image denoising, image deblurring and image inpainting problems via total variation regularization.

One of the iterative schemes to which we compare our algorithms is a primal-dual splitting method for solving highly structured monotone inclusions, as well, and it was provided by V\~u in \cite{Vu11}. Here, instead of monotone Lipschitzian operators, cocoercive operators were used and, consequently, instead of Tseng's splitting, the forward-backward splitting method has been used. The primal-dual method due to Chambolle and Pock described in \cite[Algorithm 1]{ChaPoc11} is a particular instance of V\~u's algorithm.

\section{Convex minimization problems}\label{sectionFval}
The aim of this section is to provide a rate of convergence for the sequence of the values of the objective function at the iterates generated by the algorithm
\eqref{A0} when solving a convex minimization problem and its conjugate dual. The primal-dual pair under investigation is described in the following.

\begin{problem}\label{opt-problem:general}
Consider the real Hilbert spaces $\h$ and $\g_i, i=1,...,m,$ $f\in \Gamma(\h)$ and $h : \h \rightarrow \R$ a convex and differentiable function with $\mu$-Lipschitzian gradient for some $\mu \in \R_{++}$. Furthermore, let $z \in \h$ and for every $i\in\{1,\ldots,m\}$, let $r_i \in \g_i$, $g_i, l_i \in \Gamma(\g_i)$ such that $l_i$ is $\nu_i^{-1}$-strongly convex for some $\nu_i \in \R_{++}$, and let $L_i : \h \rightarrow \g_i$ be a nonzero linear continuous operator. We consider the convex minimization problem
\begin{align}
	\label{opt-problem:general-primal}
	(P) \quad \inf_{x \in \h}{\left\{f(x)+\sum_{i=1}^m (g_i \Box l_i)(L_ix-r_i) +h(x) -\<x,z\>\right\}}
\end{align}
and its dual problem
\begin{align}
	\label{opt-problem:general-dual}
	(D) \quad \sup_{(v_i,\ldots,v_m) \in \g_1\times\ldots\times\g_m}{\left\{-\left( f^*\Box h^*\right)\left( z - \sum_{i=1}^m L_i^*v_i\right) - \sum_{i=1}^m \left( g_i^*(v_i) + l_i^*(v_i) + \< v_i,r_i \> \right) \right\} }.
\end{align}
\end{problem}
In order to investigate the primal-dual pair \eqref{opt-problem:general-primal}-\eqref{opt-problem:general-dual} in the context of Problem \ref{opt-problem:general}, one has to take
$$ A = \partial f, \ C=\nabla h, \text{ and, for }i=1,\ldots,m, \ B_i=\partial g_i \text{ and } D_i = \partial l_i.$$
Then $A$ and $B_i, i=1,...,m$ are maximal monotone, $C$ is monotone, by \cite[Proposition 17.10]{BauschkeCombettes11}, and $D_i^{-1} = \nabla l_i^*$ is monotone and $\nu_i$-Lipschitz continuous for $i=1,\ldots,m$, according to  \cite[Proposition 17.10, Theorem 18.15 and Corollary 16.24]{BauschkeCombettes11}. One can easily see that (see, for instance, \cite[Theorem 4.2]{ComPes12}) whenever $(\bx, \bv_1,\ldots,\bv_m) \in \h \times \fG$ is a primal-dual solution to Problem \ref{opt-problem-inclusion-full}, with the above choice of the involved operators, $\bx$ is an optimal solution to $(P)$, $(\bv_1,\ldots,\bv_m)$ is an optimal solution to $(D)$ and for $(P)$-$(D)$ strong duality holds, thus the optimal objective values of the two problems coincide.

The primal-dual pair in Problem \ref{opt-problem:general} captures various different types of optimization problems. One such particular instance is formulated as follows and we refer for more examples to \cite{ComPes12}.
\begin{example}
In Problem \ref{opt-problem:general} take $z=0$, let $l_i:\g_i \rightarrow \oR$, $l_i=\delta_{\{0\}}$ and $r_i=0$ for $i=1,\ldots,m$, and set $h:\h \rightarrow \R$, $h(x)=0$ for all $x\in \h$.  Then \eqref{opt-problem:general-primal} reduces to
$$ (P) \quad \inf_{x \in \h}{\left\{f(x)+\sum_{i=1}^m g_i(L_ix) \right\}}, $$
while the dual problem \eqref{opt-problem:general-dual} becomes
$$ (D) \quad \sup_{(v_i,\ldots,v_m) \in \g_1\times\ldots\times\g_m}{\left\{-f^*\left(- \sum_{i=1}^m L_i^*v_i\right) - \sum_{i=1}^m g_i^*(v_i)  \right\}}.$$
\end{example}

In order to simplify the upcoming formulations and calculations we introduce the following more compact notations. With respect to Problem \ref{opt-problem:general}, let $\ff:\h \rightarrow \oR,\ \ff(x) = f(x) + h(x) -\<x,z\>$. Then $\dom \ff = \dom f$ and its conjugate $\ff^* :\h \rightarrow \oR$ is given by $\ff^*(p) = (f + h)^*(z+p) = (f^* \Box h^*)(z+p)$, since $\dom h =\h$. Further, we set
\begin{align*}
	\fv = (v_1,\ldots,v_m), \quad	\fbv = (\bv_1,\ldots,\bv_m), \quad \fp_{2,n} = (p_{2,1,n},\ldots,p_{2,m,n}), \quad	\fr = (r_1,\ldots,r_m).
\end{align*}
We define the function $\fg:\fG \rightarrow \oR$, $\fg(\fy) = \sum_{i=1}^m (g_i\Box l_i)(y_i)$ and observe that its conjugate $\fg^* :\fG \rightarrow \oR$ is given by $\fg^*(\fv) = \sum_{i=1}^m (g_i\Box l_i)^*(v_i)=\sum_{i=1}^m(g_i^* + l_i^*)(v_i)$. Notice that, as $l_i^*, i=1,\ldots,m$, has full domain (cf. \cite[Theorem 18.15]{BauschkeCombettes11}), we get
\begin{align}
		\dom \fg^* &= (\dom g_1^* \cap \dom l_1^*) \times \ldots \times (\dom g_m^* \cap \dom l_m^*) = \dom g_1^* \times \ldots \times \dom g_m^*,
\end{align}

The primal and the dual optimization problems given in Problem \ref{opt-problem:general} can be equivalently represented as
\begin{align*}
	(P) \quad \inf_{x \in \h}{\left\{\ff(x)+\fg(\fL x -\fr)\right\}},
\end{align*}
and, respectively,
\begin{align*}
	(D) \quad \sup_{\fv \in \fG}{\left\{-\ff^*(-\fL^* \fv)-\fg^*(\fv) - \<\fv,\fr \> \right\}}.
\end{align*}
Then $\bx \in \h$ solves $(P)$, $\fbv \in \fG$ solves $(D)$ and for $(P)$-$(D)$ strong duality holds if and only if (cf. \cite{Bot10,BotGradWanka09})
\begin{align}
		\label{opt-problem-optimality-condition}
		-\fL^*\fbv &\in \partial \ff(\bx) \ \mbox{and} \	\fL \bx - \fr \in \partial \fg^*(\fbv).
\end{align}
Let us mention also that for $\bx \in \h$ and $\fbv \in \fG$ fulfilling \eqref{opt-problem-optimality-condition} it holds
$$  \left[\< \fL x-\fr, \fbv \> + \ff(x) -\fg^*(\fbv) \right]- \left[ \< \fL \bx-\fr, \fv \> + \ff(\bx) -\fg^*(\fv) \right] \geq 0 \ \forall x\in \h \ \forall \fv \in \fG.$$

For given sets $B_1\subseteq\h$ and $B_2 \subseteq \fG$ we introduce the so-called \textit{primal-dual gap function}
\begin{align}
		\label{primal-dual-gap}
		\mathcal{G}_{B_1 \times B_2}(x,\fv) &= \sup_{\tilde{\fv}\in B_2}\left\{ \< \fL x-\fr, \tilde{\fv} \> + \ff(x) -\fg^*(\tilde{\fv}) \right\} \notag \\
		& \quad- \inf_{\tilde{x} \in B_1}{\left\{ \< \fL \tilde{x}-\fr, \fv \> + \ff(\tilde{x}) -\fg^*(\fv) \right\}}.
\end{align}
We consider the following algorithm for solving $(P)$-$(D)$, which differs from the one given in Theorem \ref{theorem-inclusion-prelim} by the fact that we are asking the sequence $(\gamma_n)_{n\geq 0} \subseteq \R_{++}$ to be nondecreasing.
\begin{algorithm}
	\label{alg1}
	Let $x_0 \in \h$ and $(v_{1,0}, \ldots, v_{m,0}) \in \fG$, set
	$$ \beta = \max\{ \mu, \nu_1,\ldots, \nu_m \} + \sqrt{\sum_{i=1}^n\| L_i \|^2},$$
choose $\varepsilon \in \left(0, \frac{1}{\beta +1} \right)$ and $(\gamma_n)_{n\geq 0}$ a nondecreasing sequence in $\left[ \varepsilon, \frac{1-\varepsilon}{\beta}\right]$ and set
	\begin{align}\label{A1}
	  \left(\forall n\geq 0\right) \   \left\lfloor \begin{array}{l}
		p_{1,n} = \Prox_{\gamma_n f}\left(x_n - \gamma_n \left( \nabla h(x_n) +\sum_{i=1}^mL_i^*v_{i,n} -z\right) \right) \\
		\text{For }i=1,\ldots,m  \\
				\ \left\lfloor \begin{array}{l}
					p_{2,i,n} = \Prox_{\gamma_n g_i^*}\left(v_{i,n} +\gamma_n (L_i x_n - \nabla l_i^*(v_{i,n}) -r_i)\right) \\
					v_{i,n+1} = \gamma_n L_i( p_{1,n} - x_n) +\gamma_n(\nabla l_i^*(v_{i,n})-\nabla l_i^*(p_{2,i,n})) + p_{2,i,n}
				\end{array} \right.\\
		x_{n+1} = \gamma_n \sum_{i=1}^mL_i^*(v_{i,n}-p_{2,i,n}) + \gamma_n(\nabla h(x_n) - \nabla h(p_{1,n})) +p_{1,n}.		
		\end{array}
		\right.
	\end{align}
\end{algorithm}

\begin{theorem}
	\label{theorem-fvalues}
	For Problem \ref{opt-problem:general} suppose that
	$$ z \in \ran\left( \partial f + \sum_{i=1}^m L_i^*\left( \partial g_i \Box \partial l_i \right)(L_i\cdot -r_i) + \nabla h \right).$$
	Then there exists an optimal solution $\bx \in \h$ to $(P)$ and an optimal solution $(\bv_1,\ldots,\bv_m) \in \fG$ to $(D)$, such that the following holds for the sequences generated by Algorithm \ref{alg1}:
	\begin{enumerate}[label={(\alph*)}]
		\setlength{\itemsep}{-2pt}
        \item $z-\sum_{i=1}^mL_i^*\bv_i \in \partial f(\bx) + \nabla h(\bx)$ and $L_i \bx - r_i \in \partial g_i^*(\bv_i) + \nabla l_i^*(\bv_i)$ $\forall i\in\{1,\ldots,m\}$.
		\item $x_n \rightharpoonup \bx$, $p_{1,n} \rightharpoonup \bx$ and  $v_{i,n} \rightharpoonup \bv_i$, $p_{2,i,n} \rightharpoonup \bv_i$ $\forall i\in\{1,\ldots,m\}$.
		\item For $n\geq 0$ it holds
		$$ \frac{\| x_n - \bx \|^2}{2\gamma_n} + \sum_{i=1}^m\frac{\| v_{i,n} - \bv_i \|^2}{2\gamma_n} \leq  \frac{\| x_0 - \bx \|^2}{2\gamma_0} + \sum_{i=1}^m\frac{\| v_{i,0} - \bv_i \|^2}{2\gamma_0}. $$
		\item If $B_1\subseteq\h$ and $B_2\subseteq \fG$ are bounded, then for $x^N:=\frac{1}{N}\sum_{n=0}^{N-1} p_{1,n}$ and $v_i^N:=\frac{1}{N}\sum_{n=0}^{N-1} p_{2,i,n}$, $i=1,\ldots,m$, the primal-dual gap has the upper bound
		\begin{align} \label{fval-global-gap} \mathcal{G}_{B_1 \times B_2}(x^N,v_1^N,\ldots,v_m^N) \leq \frac{C(B_1,B_2)}{N}, \end{align}
		where
		$$ C(B_1,B_2) = \sup_{(x,v_1,\ldots,v_m) \in B_1 \times B_2}\left\{ \frac{\| x_0 - x \|^2}{2\gamma_0} + \sum_{i=1}^m\frac{\| v_{i,0} - v_i \|^2}{2\gamma_0} \right\}. $$
		\item The sequence $(x^N,v_1^N,\ldots,v_m^N)$ converges weakly to $(\bx,\bv_1,\ldots,\bv_m)$.
	\end{enumerate}
\end{theorem}

\begin{proof}
Theorem 4.2 in \cite{ComPes12} guarantees the existence of an optimal solution $\bx \in \h$ to \eqref{opt-problem:general-primal} and of an optimal solution $(\bv_1,\ldots,\bv_m) \in \fG$ to \eqref{opt-problem:general-dual} such that strong duality holds, $x_n \rightharpoonup \bx$, $p_{1,n} \rightharpoonup \bx$, as well as $v_{i,n} \rightharpoonup \bv_i$ and $p_{2,i,n} \rightharpoonup \bv_i$ for $i=1,\ldots,m$, when $n$ converges to $+\infty$. Hence \textit{(a)} and \textit{(b)} are true. Thus, the solutions $\bx$ and $\fbv=(\bv_1,\ldots,\bv_m)$ fulfill \eqref{opt-problem-optimality-condition}.

Regarding the sequences $(p_{1,n})_{n \geq 0}$ and $(p_{2,i,n})_{n \geq 0}$, $i=1,\ldots,m$, generated in Algorithm \ref{alg1} we have for every $n \geq 0$
\begin{align*}
	p_{1,n} = \left( \id + \gamma_n \partial f\right)^{-1}&\left(x_n - \gamma_n \left( \nabla h(x_n) +L^*v_n -z\right) \right)  \\
	&\Leftrightarrow  \frac{x_n-p_{1,n}}{\gamma_n} -\nabla h(x_n) - L^*v_n +z \in \partial f(p_{1,n})
\end{align*}
and, for $i=1,...,m$,
\begin{align*}
	p_{2,i,n} = \left(\id + \gamma_n \partial g_i^*\right)^{-1}&\left(v_{i,n} +\gamma_n (L_i x_n - \nabla l_i^*(v_{i,n}) -r_i)\right)  \\
	&\Leftrightarrow  \frac{v_{i,n}-p_{2,i,n}}{\gamma_n} + L_i x_n - \nabla l_i^*(v_{i,n}) -r_i \in \partial g_i^*(p_{2,i,n}).
\end{align*}
In other words, it holds for every $n \geq 0$
\begin{align}
	\label{sub-ineq-f}
	f(x) &\geq f(p_{1,n}) + \< \frac{x_n-p_{1,n}}{\gamma_n} -\nabla h(x_n) - L^*v_n +z, x-p_{1,n} \> \, \forall x\in \h
\end{align}	
and, for $i=1,\ldots,m$,
\begin{align}\label{sub-ineq-g_i^*}
	g_i^*(v_i) &\geq g_i^*(p_{2,i,n}) + \< \frac{v_{i,n}-p_{2,i,n}}{\gamma_n} + L_i x_n - \nabla l_i^*(v_{i,n}) -r_i, v_i - p_{2,i,n} \> \, \forall v_i \in \g_i.
\end{align}
In addition to that, using that $h$ and $l_i^*, i=1,...,m$, are convex and differentiable, it holds for every $n \geq 0$
\begin{align}
	\label{sub-ineq-h}
	h(x) &\geq h(p_{1,n}) + \< \nabla h(p_{1,n}), x-p_{1,n} \> \, \forall x\in \h
\end{align}
and, for $i=1,\ldots,m$,
\begin{align}
	\label{sub-ineq-l_i^*}
	l_i^*(v_i) &\geq l_i^*(p_{2,i,n}) + \< \nabla l_i^*(p_{2,i,n}), v_i - p_{2,i,n} \> \, \forall v_i \in \g_i.
\end{align}
Consider arbitrary $x \in \h$ and $\fv=(\fv_1,\ldots,\fv_m) \in \fG$. Since
\begin{align*}
	\< \frac{x_n-p_{1,n}}{\gamma_n}, x-p_{1,n} \> &= \frac{\| x_n-p_{1,n}\|^2}{2\gamma_n} + \frac{\| x-p_{1,n}\|^2}{2\gamma_n} -\frac{\| x_n-x\|^2}{2\gamma_n} \\
	\< \frac{v_{i,n}-p_{2,i,n}}{\gamma_n} , v_i - p_{2,i,n} \> &= \frac{\| v_{i,n}-p_{2,i,n}\|^2}{2\gamma_n} + \frac{\| v_i-p_{2,i,n}\|^2}{2\gamma_n} -\frac{\| v_{i,n}-v_i\|^2}{2\gamma_n}, i=1,...,m,
\end{align*}
we obtain for every $n \geq 0$, by using the more compact notation of the elements in $\fG$ and by summing up the inequalities \eqref{sub-ineq-f}--\eqref{sub-ineq-l_i^*},
\begin{align*}
	\frac{\| x_n-x\|^2}{2\gamma_n} + \frac{\| \fv_{n}- \fv\|^2}{2\gamma_n} \geq
	\frac{\| x_n-p_{1,n}\|^2}{2\gamma_n}  + \frac{\| x-p_{1,n}\|^2}{2\gamma_n} + \frac{\| \fv_{n}-\fp_{2,n}\|^2}{2\gamma_n}+ \frac{\| \fv-\fp_{2,n}\|^2}{2\gamma_n} \\
	+ \sum_{i=1}^m \< L_i x_n +\nabla l_i^*(p_{2,i,n}) - \nabla l_i^*(v_{i,n}) -r_i, v_i - p_{2,i,n} \> - \sum_{i=1}^m (g_i^*+l_i^*)(v_i) +(f+h)(p_{1,n})  \\
	+ \< \nabla h(p_{1,n})-\nabla h(x_n) - L^*v_n +z, x-p_{1,n} \> - \left[ \sum_{i=1}^m-(g_i^* +l_i^*)(p_{2,i,n}) + (f+h)(x) \right].
\end{align*}
Further, using again the update rules in Algorithm \ref{alg1} and the equations
\begin{align*}
	\< \frac{p_{1,n}-x_{n+1}}{\gamma_n}, x-p_{1,n} \> &= \frac{\| x_{n+1}-x\|^2}{2\gamma_n} - \frac{\| x_{n+1} -p_{1,n}\|^2}{2\gamma_n} -\frac{\| x-p_{1,n}\|^2}{2\gamma_n}
\end{align*}
and, for $i=1,...,m$,
\begin{align*}
	\< \frac{p_{2,i,n}-v_{i,n+1}}{\gamma_n} , v_i - p_{2,i,n} \> &= \frac{\| v_{i,n+1}-v_i\|^2}{2\gamma_n} - \frac{\| v_{i,n+1}-p_{2,i,n}\|^2}{2\gamma_n} -\frac{\| v_i-p_{2,i,n}\|^2}{2\gamma_n},
\end{align*}
we obtain for every $n \geq 0$
\begin{align}
	\label{ineq-fval1}
	\frac{\| x_n-x\|^2}{2\gamma_n} &+ \frac{\| \fv_{n}- \fv\|^2}{2\gamma_n} 	\geq	\frac{\|x_{n+1}-x \|^2}{2\gamma_n}  + \frac{\| \fv_{n+1}-\fv\|^2}{2\gamma_n}+\frac{\| x_n-p_{1,n}\|^2}{2\gamma_n} +  \frac{\| \fv_{n}-\fp_{2,n}\|^2}{2\gamma_n}    \notag \\
	& -\frac{\| x_{n+1}-p_{1,n} \|^2}{2\gamma_n} - \frac{\| \fv_{n+1} -\fp_{2,n} \|^2}{2\gamma_n} + \left[ \< \fL p_{1,n} -\fr,\fv \> - \fg^*(\fv) + \ff (p_{1,n})\right] \notag \\
	&- \left[ \< \fL x -\fr,\fp_{2,n}\> - \fg^*(\fp_{2,n}) + \ff (x)\right].
\end{align}
Further, we equip the Hilbert space $\fH=\h \times \fG$ with the inner product
\begin{equation}\label{inprodfH}
\langle (y,\fp),(z,\fq)\rangle = \langle y,z\rangle + \langle \fp,\fq\rangle \ \forall (y,\fp), (z,\fq) \in \h \times \fG
\end{equation}
and the associated norm $\|(y,\fp)\| = \sqrt{\|y\|^2 + \|\fp\|^2}$ for every $(y,\fp) \in \h \times \fG$.
For every $n \geq 0$ it holds
\begin{align*}
	\frac{\| x_{n+1}-p_{1,n} \|^2}{2\gamma_n}  + \frac{\| \fv_{n+1} -\fp_{2,n} \|^2}{2\gamma_n}
	= \frac{\| (x_{n+1},\fv_{n+1})-(p_{1,n},\fp_{2,n}) \|^2}{2\gamma_n}
\end{align*}
and, consequently, by making use of the Lipschitz continuity of $\nabla h$ and $\nabla l_i^*$, $i=1,\ldots,m$, it shows that
{\allowdisplaybreaks
\begin{align}
	&\| (x_{n+1},\fv_{n+1})-(p_{1,n},\fp_{2,n}) \| \notag\\
	&= \gamma_n \| ( \fL^*(\fv_n-\fp_{2,n}) , L_1( p_{1,n} - x_n), \ldots, L_m(p_{1,n}-x_n)  ) \notag\\
	& \qquad \quad + (\nabla h(x_n) - \nabla h(p_{1,n}), \nabla l_1^*(v_{1,n})-\nabla l_1^*(p_{2,1,n}), \ldots, \nabla l_m^*(v_{m,n})-\nabla l_1^*(p_{2,m,n}))\| \notag\\
	&\leq \gamma_n \| ( \fL^*(\fv_n-\fp_{2,n}) , L_1( p_{1,n} - x_n), \ldots, L_m(p_{1,n}-x_n)  ) \| \notag\\
	& \quad + \gamma_n \|(\nabla h(x_n) - \nabla h(p_{1,n}), \nabla l_1^*(v_{1,n})-\nabla l_1^*(p_{2,1,n}), \ldots, \nabla l_m^*(v_{m,n})-\nabla l_1^*(p_{2,m,n}))\| \notag\\
	&= \gamma_n \sqrt{ \left\| \sum_{i=1}^m L_i^*(v_{i,n}-p_{2,i,n}) \right\|^2 + \sum_{i=1}^m \left\| L_i( p_{1,n} - x_n) \right\|^2} \notag\\
	& \quad + \gamma_n \sqrt{ \| \nabla h(x_n) - \nabla h(p_{1,n}) \|^2 + \sum_{i=1}^m \| \nabla l_i^*(v_{i,n})-\nabla l_i^*(p_{2,i,n}) \|^2} \notag\\
	&\leq \gamma_n \sqrt{  \left( \sum_{i=1}^m \| L_i\|^2 \right) \sum_{i=1}^m \left\| v_{i,n}-p_{2,i,n} \right\|^2 + \left( \sum_{i=1}^m \| L_i \|^2 \right) \left\|  p_{1,n} - x_n \right\|^2} \notag\\
	& \quad + \gamma_n \sqrt{ \mu^2 \| x_n - p_{1,n} \|^2 + \sum_{i=1}^m \nu_i^2 \| v_{i,n}-p_{2,i,n} \|^2} \notag\\
		\label{ineq-lipschitz-continuity}
	&\leq \gamma_n \left( \sqrt{\sum_{i=1}^m \| L_i \|^2} + \max\{ \mu, \nu_1,\ldots, \nu_m \} \right) \| (x_n,\fv_n) - (p_{1,n},\fp_{2,n}) \|.
\end{align}}

Hence, by taking into consideration the way in which $(\gamma_n)_{n\geq 0}$ is chosen, we have for every $n \geq 0$
\begin{align*}
	&\frac{1}{2\gamma_n} \left[ \| x_n-p_{1,n}\|^2 +  \| \fv_{n}-\fp_{2,n}\|^2 -\| x_{n+1}-p_{1,n} \|^2 - \| \fv_{n+1} -\fp_{2,n} \|^2 \right] \\
	&\geq \frac{1}{2\gamma_n} \left(1-\gamma_n^2 \left( \sqrt{\sum_{i=1}^m \| L_i \|^2} + \max\{ \mu, \nu_1,\ldots, \nu_m \} \right)^2\right) \| (x,\fv_n) - (p_{1,n},\fp_{2,n}) \|^2 \geq 0.
\end{align*}
and, consequently, \eqref{ineq-fval1} reduces to
\begin{align*}
	\frac{\| x_n-x\|^2}{2\gamma_n} + \frac{\| \fv_{n}- \fv\|^2}{2\gamma_n} &\geq	\frac{\gamma_{n+1}}{\gamma_n}\frac{\|x_{n+1}-x \|^2}{2\gamma_{n+1}} + \left[ \< \fL p_{1,n} -\fr,\fv \> - \fg^*(\fv) + \ff (p_{1,n})\right]  \\ &\quad + \frac{\gamma_{n+1}}{\gamma_n}\frac{\| \fv_{n+1}-\fv\|^2}{2\gamma_{n+1}} - \left[ \< \fL x -\fr,\fp_{2,n}\> - \fg^*(\fp_{2,n}) + \ff (x)\right].
\end{align*}
Let $N \geq 1$ be an arbitrary natural number. Summing the above inequality up from $n=0$ to $N-1$ and using the fact that $(\gamma_n)_{n\geq 0}$ is nondecreasing, it follows that
\begin{align}
	\label{ineq-sum-fval}
	\frac{\| x_0-x\|^2}{2\gamma_0} + \frac{\| \fv_{0}- \fv\|^2}{2\gamma_0}
	&\geq	\frac{\|x_{N}-x \|^2}{2\gamma_{N}} + \sum_{n=0}^{N-1}\left[ \< \fL p_{1,n} -\fr,\fv \> - \fg^*(\fv) + \ff (p_{1,n})\right] \notag\\
	&\quad + \frac{\| \fv_{N}-\fv\|^2}{2\gamma_{N}}- \sum_{n=0}^{N-1}\left[ \< \fL x -\fr,\fp_{2,n}\> - \fg^*(\fp_{2,n}) + \ff (x)\right].
\end{align}

Replacing $x=\bx$ and $\fv = \fbv$ in the above estimate, since they fulfill \eqref{opt-problem-optimality-condition}, we obtain
$$\sum_{n=0}^{N-1}\left[ \< \fL p_{1,n} -\fr,\boldsymbol{\bv} \> - \fg^*(\boldsymbol{\bv}) + \ff (p_{1,n})\right] - \sum_{n=0}^{N-1}\left[ \< \fL \bx -\fr,\fp_{2,n}\> - \fg^*(\fp_{2,n}) + \ff (\bx)\right] \geq 0.$$
Consequently,
$$ \frac{\| x_0-\bx\|^2}{2\gamma_0} + \frac{\| \fv_{0}- \boldsymbol{\bv}\|^2}{2\gamma_0} \geq \frac{\|x_{N}-\bx \|^2}{2\gamma_{N}}+\frac{\| \fv_{N}-\boldsymbol{\bv}\|^2}{2\gamma_{N}}$$
and statement \textit{(c)} follows.
On the other hand, dividing \eqref{ineq-sum-fval} by $N$, using the convexity of $\ff$ and $\fg^*$, and denoting $x^N:=\frac{1}{N}\sum_{n=0}^{N-1} p_{1,n}$ and $v_i^N:=\frac{1}{N}\sum_{n=0}^{N-1} p_{2,i,n}$, $i=1,\ldots,m$, we obtain
\begin{align*}
	\frac{1}{N}\left(\frac{\| x_0-x\|^2}{2\gamma_0} + \frac{\| \fv_{0}- \fv\|^2}{2\gamma_0} \right)
	&\geq	\left[ \< \fL x^N -\fr,\fv \> - \fg^*(\fv) + \ff (x^N)\right] \\ &\quad - \left[ \< \fL x -\fr,\fv^N\> - \fg^*(\fv^N) + \ff (x)\right], \notag
\end{align*}
which shows \eqref{fval-global-gap} when passing to the supremum over $x\in B_1$ and $\fv \in B_2$. In this way statement \textit{(d)} is verified. The weak convergence of $(x^N,\fv^N)$ to $(\bx,\boldsymbol{\bv})$ when $N$ converges to $+\infty$ is an easy consequence of the Stolz--Ces{\`a}ro Theorem, fact which shows \textit{(e)}.
\end{proof}

\begin{remark}\label{remark1}
In the situation when the functions $g_i$ are Lipschitz continuous on $\g_i, i=1,...,m,$ inequality \eqref{fval-global-gap} provides for the sequence of the values of the objective of $(P)$ taken at $(x^N)_{N \geq 1}$ a convergence rate of $\O(\frac{1}{N})$, namely, it holds
\begin{align}
		\label{ineq-primal-convergence}
		\ff (x^N) + \fg(\fL x^N -\fr) - \ff(\bx) -\fg(\fL \bx-\fr) \leq \frac{C(B_1,B_2)}{N} \ \forall N \geq 1.
\end{align}
Indeed, due to statement \textit{(b)} of the previous theorem, the sequence $(p_{1,n})_{n\geq 0} \subseteq \h$ is bounded and one can take $B_1 \subset \h$ being a bounded, convex and closed set containing this sequence. Obviously, $\bar x \in B_1$. On the other hand, we take $B_2=\dom g_1^* \times \ldots \times \dom g_m^*$, which is in this situation a bounded set. Then it holds, using the Fenchel-Moreau Theorem and the Young-Fenchel inequality, that
\begin{align*}
		\mathcal{G}_{B_1 \times B_2}(x^N,\fv^N) &= \ff(x^N) + \fg(\fL x^N -\fr) + \fg^*(\fv^N) - \inf_{\tilde{x} \in B_1}{\left\{ \< \fL \tilde{x}-\fr, \fv^N \> + \ff(\tilde{x}) \right\}} \\
		&\geq \ff(x^N) + \fg(\fL x^N -\fr) + \fg^*(\fv^N)  -\< \fL \bx-\fr, \fv^N \> - \ff(\bx) \\
		&\geq \ff(x^N) + \fg(\fL x^N -\fr) - \ff(\bx) -\fg(\fL \bx-\fr).
\end{align*}
Hence, \eqref{ineq-primal-convergence} follows by statement \textit{(d)} in Theorem \ref{theorem-fvalues}.

In a similar way, one can show that, whenever $f$ is Lipschitz continuous, \eqref{fval-global-gap} provides for the sequence of the values of the objective of $(D)$ taken at $(v^N)_{N \geq 1}$ a convergence rate of $\O(\frac{1}{N})$.
\end{remark}

\begin{remark}\label{remark11}
If $\g_i$, $i=1,\ldots,m$, are finite-dimensional real Hilbert spaces, then \eqref{ineq-primal-convergence} is true, even under the weaker assumption that the convex functions $g_i, i=1,...,m$, have full domain, without necessarily being Lipschitz continuous. The set $B_1 \subset \h$ can be chosen as in Remark \ref{remark1}, but this time we take $B_2 = \bigtimes _{i=1}^m \bigcup_{n\geq 0}\partial g_i\left( L_i p_{1,n}\right) \subset \fG$, by noticing also that the functions $g_i, i=1,...,m$, are everywhere subdifferentiable.

The set $B_2$ is bounded, as for every $i=1,\ldots,m$ the set $\bigcup_{n\geq 0}\partial g_i\left( L_i p_{1,n}\right)$ is bounded. Let be $i\in \{1,...,m\}$ fixed. Indeed, as $p_{1,n} \rightharpoonup \bx$, it follows that $L_i p_{1,n} \rightarrow L_i\bar x$ for $i=1,...,m$. Using the fact that the subdifferential of $g_i$ is a locally bounded operator at $L_i \bar x$, the boundedness of $\bigcup_{n\geq 0}\partial g_i\left( L_i p_{1,n}\right)$ follows automatically.

For this choice of the sets $B_1$ and $B_2$, by using the same arguments as in the previous remark, it follows that \eqref{ineq-primal-convergence} is true.
\end{remark}

\section{Zeros of sums of monotone operators}\label{sectionOperator}
In this section we turn our attention to the primal-dual monotone inclusion problems formulated in Problem \ref{opt-problem-inclusion-full} with the aim to provide accelerations of the iterative method described in Theorem \ref{theorem-inclusion-prelim} under the additional strong monotonicity assumptions.

\subsection{The case when \texorpdfstring{$A+C$}{A+C} is strongly monotone}

We focus first on the case when $A+C$ is $\rho$-strongly monotone for some $\rho \in \R_{++}$ and investigate the impact of this assumption on the convergence rate of the sequence of primal iterates. The condition $A+C$ is $\rho$-strongly monotone is fulfilled when either $A:\h \rightarrow 2^{\h}$ or $C:\h \rightarrow \h$ is $\rho$-strongly monotone. In case that $A$ is $\rho_1$-monotone and $C$ is $\rho_2$-monotone, the sum $A+C$ is $\rho$-monotone with $\rho=\rho_1+\rho_2$.

\begin{remark}
\label{rm1}
The situation when $B_i^{-1} + D_i^{-1}$ is $\tau_i$-strongly monotone with $\tau_i \in \R_{++}$ for $i=1,\ldots,m$, which improves the convergence rate of the sequence of dual iterates, can be handled with appropriate modifications.
\end{remark}

Due to technical reasons we assume in the following that the operators $D_i^{-1}$ in Problem \ref{opt-problem-inclusion-full} are zero for $i=1,\ldots,m$, thus, $D_i(0) = \g_i$ and $D_i(x) = \emptyset$ for $x \neq 0$, for $i=1,...,m$. In Remark \ref{rm2} we show how the results given in this particular context can be employed when treating the primal-dual pair of monotone inclusions \eqref{opt-problem-primal-inclusion-full}-\eqref{opt-problem-dual-inclusion-full}. Consequently, the problem we deal with in this subsection is as follows.

\begin{problem}\label{opt-problem-inclusion}
Consider the real Hilbert spaces $\h$ and $\g_i, i=1,...,m,$ $A:\h \rightarrow 2^{\h}$ a maximally monotone operator and $C:\h \rightarrow \h$ a monotone and $\mu$-Lipschitzian operator for some $\mu \in \R_{++}$. Furthermore, let $z \in \h$ and for every $i\in\{1,\ldots,m\}$, let $r_i \in \g_i$, let $B_i : \g_i \rightarrow 2^{\g_i}$ be maximally monotone operators and let $L_i : \h \rightarrow \g_i$ be a nonzero linear continuous operator. The problem is to solve the primal inclusion
\begin{align}
	\label{opt-problem-primal-inclusion}
	\text{find }\bx \in \h \text{ such that } z \in A\bx + \sum_{i=1}^m L_i^* B_i(L_i \bx-r_i) + C\bx,
\end{align}
together with the dual inclusion
\begin{align}
	\label{opt-problem-dual-inclusion}
	\text{find }\bv_1 \in \g_1,\ldots,\bv_m \in \g_m \text{ such that }(\exists x\in\h)\left\{
	\begin{array}{l}
		z - \sum_{i=1}^m L_i^*\bv_i \in Ax + Cx \\
		\bv_i \in B_i(L_ix-r_i), \,i=1,\ldots,m.
	\end{array}
\right.
\end{align}
\end{problem}
The subsequent algorithm represents an accelerated version of the one given in Theorem \ref{theorem-inclusion-prelim} and relies on the fruitful idea of using
a second sequence of variable step length parameters $(\sigma_n)_{n \geq 0} \subseteq \R_{++}$, which, together with the sequence of parameters $(\gamma_n)_{n \geq 0} \subseteq \R_{++}$, play an important role in the convergence analysis.
\begin{algorithm}
	\label{alg2}
	Let $x_0 \in \h$, $(v_{1,0}, \ldots, v_{m,0}) \in \fG$,
	$$	\gamma_0 \in  \left(0,\min\left\{1, \frac{\sqrt{1+4\rho}}{2(1+2\rho)\mu}\right\}\right)  \text{ and set }
	  \sigma_0 =  \frac{1}{2\gamma_0(1+2\rho)\sum_{i=1}^m \|L_i \|^2}.$$
	Consider the following updates:
	\begin{align}\label{A2}
	  \left(\forall n\geq 0\right) \   \left\lfloor \begin{array}{l}
		p_{1,n} = J_{\gamma_n A}\left(x_n - \gamma_n \left( Cx_n +\sum_{i=1}^mL_i^* v_{i,n} -z\right) \right) \\
		\text{For }i=1,\ldots,m  \\
				\ \left\lfloor \begin{array}{l}
					p_{2,i,n} = J_{\sigma_n B_i^{-1}}\left(v_{i,n} +\sigma_n (L_i x_n  -r_i)\right) \\
					v_{i,n+1} = \sigma_n L_i( p_{1,n} - x_n)  + p_{2,i,n} \\
				\end{array} \right.\\
		x_{n+1} = \gamma_n \sum_{i=1}^m L_i^*(v_{i,n}-p_{2,i,n}) + \gamma_n(C x_n - C p_{1,n}) +p_{1,n} \\
		\theta_n=1/\sqrt{1+2\rho\gamma_n(1-\gamma_n)}, \ \gamma_{n+1} = \theta_n\gamma_n, \ \sigma_{n+1} = \sigma_n/\theta_n.		
		\end{array}
		\right.
	\end{align}
\end{algorithm}
\begin{theorem}\label{theorem-inclusion}
In Problem \ref{opt-problem-inclusion} suppose that $A+C$ is $\rho$-strongly monotone with $\rho \in \R_{++}$ and let $(\bx,\bv_1,\ldots,\bv_m) \in \h\times\fG$ be a primal-dual solution to Problem \ref{opt-problem-inclusion}. Then for every $n\geq 0$ it holds
 \begin{align}
	\label{ineq-th2}
	\| x_{n} - \bx \|^2 + \gamma_n \sum_{i=1}^m \frac{\| v_{i,n} - \bv_i \|^2}{\sigma_{n}} \leq \gamma_n^2 \left( \frac{\| x_0 - \bx \|^2}{\gamma_0^2} + \sum_{i=1}^m\frac{\| v_{i,0} - \bv_i \|^2}{\gamma_0\sigma_0}\right),
\end{align}
where $\gamma_n,\,\sigma_n \in \R_{++}$, $x_n \in \h$ and $(v_{1,n},\ldots,v_{m,n}) \in \fG$ are the iterates generated by Algorithm \ref{alg2}.
\end{theorem}

\begin{proof}
Taking into account the definitions of the resolvents occurring in Algorithm \ref{alg2} we obtain
\begin{align*}
	\text{ and }\ \left.
	\begin{aligned} 	
		\frac{x_n-p_{1,n}}{\gamma_n} -Cx_n - \sum_{i=1}^m L_i^*v_{i,n} +z &\in A p_{1,n} \\
		\frac{v_{i,n}-p_{2,i,n}}{\sigma_{n}} + L_i x_n -r_i &\in B_i^{-1} p_{2,i,n}, i=1,\ldots,m,
	\end{aligned}
	\right.
\end{align*}
which, in the light of the updating rules in \eqref{A2}, furnishes for every $n \geq 0$
\begin{align}
	\label{inclusion-alg2-prelim}
	\text{ and } \ \left.
	\begin{aligned}
		\frac{x_n-x_{n+1}}{\gamma_n} - \sum_{i=1}^m L_i^*p_{2,i,n} +z &\in (A+C) p_{1,n} \\
		\frac{v_{i,n}-v_{i,n+1}}{\sigma_{n}} + L_i p_{1,n}  -r_i &\in B_i^{-1} p_{2,i,n}, i=1,\ldots,m.
	\end{aligned}
	\right.
\end{align}
The primal-dual solution $(\bx,\bv_1,\ldots,\bv_m) \in \h\times\fG$ to Problem \ref{opt-problem-inclusion} fulfills (see \eqref{operator-proof-conditions-full}, where $D_i^{-1}$ are taken to be zero for $i=1,...,m$)
\begin{equation*}
z - \sum_{i=1}^m L_i^*\bv_i \in A\bx + C\bx \ \mbox{and} \  \bv_i \in B_i(L_i\bx-r_i), \,i=1,\ldots,m.
\end{equation*}
Since the sum $A+C$ is $\rho$-strongly monotone, we have for every $n \geq 0$
\begin{align}
	\label{operator-monotone1}
	\< p_{1,n} - \bx, \frac{x_n-x_{n+1}}{\gamma_n} - \sum_{i=1}^m L_i^*p_{2,i,n} +z -\left( z - \sum_{i=1}^m L_i^*\bv_i \right) \> \geq \rho \| p_{1,n} - \bx \|^2
\end{align}
while, due to the monotonicity of $B_i^{-1}:\g_i \rightarrow 2^{\g_i}$, we obtain for every $n \geq 0$
\begin{align}
	\label{operator-monotone3}
	\< p_{2,i,n} - \bv_i , \frac{v_{i,n}-v_{i,n+1}}{\sigma_{n}} + L_i p_{1,n} -r_i - \left( L_i \bx  -r_i \right)\> \geq 0, \ i=1,\ldots,m.
\end{align}
Further, we set
$$ \fbv = (\bv_1,\ldots,\bv_m), \quad \fv_n = (v_{1,n},\ldots,v_{m,n}), \quad \fp_{2,n}=(p_{2,1,n},\ldots,p_{2,m,n}). $$
Summing up the inequalities \eqref{operator-monotone1} and \eqref{operator-monotone3}, it follows that
\begin{align}
	\label{ineq-sum-monotone}
	&\< p_{1,n} - \bx, \frac{x_n-x_{n+1}}{\gamma_n} \> + \< \fp_{2,n} - \fbv , \frac{\fv_{n}-\fv_{n+1}}{\sigma_{n}} \> + \< p_{1,n} - \bx, \fL^*(\fbv-\fp_{2,n})\> \notag\\ &+ \< \fp_{2,n} - \fbv, \fL(p_{1,n}-\bx) \>  \geq \rho \| p_{1,n} - \bx \|^2.
\end{align}
and, from here,
\begin{align}
	\label{ineq-monotone}
	\< p_{1,n} - \bx, \frac{x_n-x_{n+1}}{\gamma_n} \> + \< \fp_{2,n} - \fbv , \frac{\fv_{n}-\fv_{n+1}}{\sigma_{n}} \>  \geq \rho \| p_{1,n} - \bx \|^2 \ \forall n \geq 0.
\end{align}
In the light of the equations
\begin{align*}
	\< p_{1,n} - \bx, \frac{x_n - x_{n+1}}{\gamma_n} \>&= \< p_{1,n} - x_{n+1}, \frac{x_n - x_{n+1}}{\gamma_n} \> + \< x_{n+1} - \bx, \frac{x_n - x_{n+1}}{\gamma_n} \> \\
	&= \frac{\| x_{n+1} - p_{1,n} \|^2}{2\gamma_n} - \frac{\|x_n- p_{1,n}\|^2}{2\gamma_n} + \frac{\| x_n - \bx \|^2}{2\gamma_n} - \frac{\| x_{n+1} - \bx \|^2}{2\gamma_n},
\end{align*}
and
\begin{align*}
	\< \fp_{2,n} - \fbv, \frac{\fv_n - \fv_{n+1}}{\sigma_n} \>&= \< \fp_{2,n} - \fv_{n+1}, \frac{\fv_n - \fv_{n+1}}{\sigma_n} \> + \< \fv_{n+1} - \fbv, \frac{\fv_n - \fv_{n+1}}{\sigma_n} \> \\
	&= \frac{\| \fv_{n+1}-\fp_{2,n}   \|^2}{2\sigma_n} - \frac{\|\fv_n-\fp_{2,n}\|^2}{2\sigma_n} + \frac{\| \fv_n - \fbv \|^2}{2\sigma_n} - \frac{\| \fv_{n+1} - \fbv \|^2}{2\sigma_n}
\end{align*}
inequality \eqref{ineq-monotone} reads for every $n\geq 0$
\begin{align}
	\frac{\| x_n - \bx \|^2}{2\gamma_n} + \frac{\| \fv_n - \fbv \|^2}{2\sigma_n} &\geq \rho \| p_{1,n} - \bx \|^2 + \frac{\| x_{n+1} - \bx \|^2}{2\gamma_n} + \frac{\| \fv_{n+1} - \fbv \|^2}{2\sigma_n} + \frac{\|x_n- p_{1,n}\|^2}{2\gamma_n} \notag\\
	\label{ineq-m1}
	&\quad + \frac{\|\fv_n-\fp_{2,n}\|^2}{2\sigma_n} - \frac{\| x_{n+1} - p_{1,n} \|^2}{2\gamma_n} - \frac{\| \fv_{n+1}-\fp_{2,n}   \|^2}{2\sigma_n}.
\end{align}
Using that $2ab \leq \alpha a^2 + \frac{b^2}{\alpha}$ for all $a,b \in \R$, $\alpha \in \R_{++}$, we obtain for $\alpha:=\gamma_n$,
\begin{align*}
	\rho \| p_{1,n} - \bx \|^2 &\geq \rho \left( \| x_{n+1} - \bx \|^2 - 2 \| x_{n+1} - \bx \| \| x_{n+1} - p_{1,n} \| + \|x_{n+1} - p_{1,n} \|^2\right) \\
	&\geq \frac{2\rho\gamma_n(1-\gamma_n)}{2\gamma_n} \| x_{n+1} - \bx \|^2 - \frac{2\rho(1-\gamma_n)}{2\gamma_n} \|x_{n+1} - p_{1,n} \|^2,
\end{align*}
which, in combination with \eqref{ineq-m1}, yields for every $n\geq 0$
\begin{align}
	&\frac{\| x_n - \bx \|^2}{2\gamma_n} + \frac{\| \fv_n - \fbv \|^2}{2\sigma_n} \geq  \frac{(1+2\rho\gamma_n(1-\gamma_n))\| x_{n+1} - \bx \|^2}{2\gamma_n} + \frac{\| \fv_{n+1} - \fbv \|^2}{2\sigma_n} \notag\\
	\label{ineq-m2}
	& + \frac{\|x_n- p_{1,n}\|^2}{2\gamma_n} + \frac{\|\fv_n-\fp_{2,n}\|^2}{2\sigma_n} - \frac{(1+2\rho(1-\gamma_n))\| x_{n+1} - p_{1,n} \|^2}{2\gamma_n} - \frac{\| \fv_{n+1}-\fp_{2,n}   \|^2}{2\sigma_n}.
\end{align}
Investigating the last two terms in the right-hand side of the above estimate it shows for every $n \geq 0$ that
\begin{align*}
&- \frac{(1+2\rho(1-\gamma_n))\| x_{n+1} - p_{1,n} \|^2}{2\gamma_n}  \\
&\geq-\frac{(1+2\rho)\gamma_n}{2}\left\| \sum_{i=1}^m L_i^*(v_{i,n}-p_{2,i,n}) + (C x_n - C p_{1,n}) \right\|^2 \\
&\geq -\frac{2(1+2\rho)\gamma_n}{2}\left( \left(\sum_{i=1}^m \|L_i \|^2\right) \|\fv_{n}-\fp_{2,n}\|^2 + \mu^2 \|x_n - p_{1,n}\|^2 \right),
\end{align*}
and
\begin{align*}
- \frac{\| \fv_{n+1}-\fp_{2,n} \|^2}{2\sigma_n}
=-\frac{\sigma_n}{2} \left( \sum_{i=1}^m \| L_i( p_{1,n} - x_n) \|^2 \right)
\geq -\frac{\sigma_n}{2} \left(\sum_{i=1}^m \|L_i\|^2\right) \| p_{1,n} - x_n\|^2 .
\end{align*}
Hence, for every $n \geq 0$ it holds
\begin{align*}
	&\frac{\|x_n- p_{1,n}\|^2}{2\gamma_n} + \frac{\|\fv_n-\fp_{2,n}\|^2}{2\sigma_n} - \frac{(1+2\rho(1-\gamma_n))\| x_{n+1} - p_{1,n} \|^2}{2\gamma_n} - \frac{\| \fv_{n+1}-\fp_{2,n}   \|^2}{2\sigma_n} \\
	&\geq \frac{\left(1 - \gamma_n\sigma_n \sum_{i=1}^m \|L_i\|^2 - 2(1+2\rho)\gamma_n^2 \mu^2 \right)}{2\gamma_n} \| p_{1,n} - x_n\|^2 \\
	&\quad + \frac{\left(1-2\gamma_n\sigma_n(1+2\rho)\sum_{i=1}^m \|L_i \|^2\right) }{2\sigma_n} \|\fv_{n}-\fp_{2,n}\|^2\\
&\geq 0.
\end{align*}
The nonnegativity of the expression in the above relation follows because of the sequence $(\gamma_n)_{n\geq 0}$ is nonincreasing, $\gamma_n \sigma_n = \gamma_0 \sigma_0$ for every $n\geq 0$ and
\begin{align*}
	\gamma_0 \in  \left(0,\min\left\{1, \frac{\sqrt{1+4\rho}}{2(1+2\rho)\mu}\right\}\right) \ \mbox{and} \ \sigma_0 =  \frac{1}{2\gamma_0(1+2\rho)\sum_{i=1}^m \|L_i \|^2}.
\end{align*}
Consequently, inequality \eqref{ineq-m2} becomes
\begin{align}
	\label{ineq-m3}
	\frac{\| x_n - \bx \|^2}{2\gamma_n} + \frac{\| \fv_n - \fbv \|^2}{2\sigma_n} \geq  \frac{(1+2\rho\gamma_n(1-\gamma_n))\| x_{n+1} - \bx \|^2}{2\gamma_{n}} + \frac{\| \fv_{n+1} - \fbv \|^2}{2\sigma_{n}} \ \forall n \geq 0.
\end{align}
Notice that we have $\gamma_{n+1} < \gamma_n$, $\sigma_{n+1}>\sigma_n$ and $\gamma_{n+1}\sigma_{n+1}=\gamma_n\sigma_n$ for every $n \geq 0$. Dividing \eqref{ineq-m3} by $\gamma_n$ and making use of
\begin{align*}
	\theta_n = \frac{1}{\sqrt{1+2\rho\gamma_n(1-\gamma_n)}}, \quad \gamma_{n+1} = \theta_n \gamma_n, \quad \sigma_{n+1}=\frac{\sigma_n}{\theta_n},
\end{align*}
 we obtain
\begin{align*}
	\frac{\| x_n - \bx \|^2}{2\gamma_n^2} + \frac{\| \fv_n - \fbv \|^2}{2\gamma_n\sigma_n} \geq  \frac{\| x_{n+1} - \bx \|^2}{2\gamma_{n+1}^2} + \frac{\| \fv_{n+1} - \fbv \|^2}{2\gamma_{n+1}\sigma_{n+1}} \ \forall n \geq 0.
\end{align*}
Let be $N \geq 1$. Summing this inequalities from $n=0$ to $N-1$, we finally get
\begin{align}
	\label{ineq-m4}
	\frac{\| x_0 - \bx \|^2}{2\gamma_0^2} + \frac{\| \fv_0 - \fbv \|^2}{2\gamma_0\sigma_0} \geq  \frac{\| x_{N} - \bx \|^2}{2\gamma_{N}^2} + \frac{\| \fv_{N} - \fbv \|^2}{2\gamma_{N}\sigma_{N}}.
\end{align}
In conclusion,
\begin{align}
	\label{ineq-m5}
	\frac{\| x_{n} - \bx \|^2}{2} + \gamma_n\frac{\| \fv_{n} - \fbv \|^2}{2\sigma_{n}} \leq \gamma_n^2 \left( \frac{\| x_0 - \bx \|^2}{2\gamma_0^2} + \frac{\| \fv_0 - \fbv \|^2}{2\gamma_0\sigma_0}\right) \ \forall n \geq 0,
\end{align}
which completes the proof.
\end{proof}
Next we show that $\rho \gamma_n$ converges like $\frac{1}{n}$ as $n \rightarrow +\infty$.
\begin{proposition}
	Let $\gamma_0 \in (0,1)$ and consider the sequence $(\gamma_n)_{n\geq 0} \subseteq \R_{++}$, where
	\begin{align}
		\label{prop-eq1}
		\gamma_{n+1} = \frac{\gamma_n}{\sqrt{1+2\rho\gamma_n(1-\gamma_n)}} \ \forall n \geq 0.
	\end{align}
	Then $\lim_{n \rightarrow +\infty} n\rho\gamma_n=1$.
\end{proposition}
\begin{proof}
Since the sequence $(\gamma_n)_{n\geq 0} \subseteq \left(0,1\right)$ is bounded and decreasing, it converges towards some $l \in \left[0,1\right)$ as $n \rightarrow +\infty$. We let $n \rightarrow + \infty$ in \eqref{prop-eq1} and obtain
\begin{align*}
		l^2 (1+2\rho l (1-l)) = l^2 \Leftrightarrow 2\rho l^3 (1-l) = 0,
\end{align*}
which shows that $l=0$, i.\,e. $\gamma_n \rightarrow 0 \ (n \rightarrow +\infty)$. On the other hand, \eqref{prop-eq1} implies that $\frac{\gamma_n}{\gamma_{n+1}} \rightarrow 1 (n \rightarrow +\infty)$.
As $(\frac{1}{\gamma_n})_{n\geq 0}$ is a strictly increasing and unbounded sequence, by applying the Stolz--Ces{\`a}ro Theorem it shows that
\begin{align*}
\lim_{n \rightarrow + \infty} n\gamma_n
	&=\lim_{n \rightarrow + \infty} \frac{n}{\frac{1}{\gamma_n}}
	= \lim_{n \rightarrow + \infty} \frac{ n+1- n}{\frac{1}{\gamma_{n+1}} - \frac{1}{\gamma_n}}
	= \lim_{n \rightarrow + \infty} \frac{ \gamma_n \gamma_{n+1}}{\gamma_n - \gamma_{n+1}}\\
	&= \lim_{n \rightarrow + \infty} \frac{ \gamma_n \gamma_{n+1} (\gamma_n + \gamma_{n+1})}{\gamma_n^2 - \gamma_{n+1}^2}
	\overset{\mathclap{\eqref{prop-eq1}}}{=} \lim_{n \rightarrow + \infty} \frac{ \gamma_n \gamma_{n+1} (\gamma_n + \gamma_{n+1})}{2\rho \gamma_{n+1}^2 \gamma_n (1-\gamma_n)}\\
	&= \lim_{n \rightarrow + \infty} \frac{ \gamma_n + \gamma_{n+1}}{2\rho \gamma_{n+1}(1-\gamma_n)}
	= \lim_{n \rightarrow + \infty} \frac{ \frac{\gamma_n}{\gamma_{n+1}} + 1}{2\rho (1-\gamma_n)}
	= \frac{2}{2\rho} = \frac{1}{\rho},
\end{align*}
which completes the proof.
\end{proof}
Hence, we have shown the following result.
\begin{theorem}\label{theorem-inclusion2}
In Problem \ref{opt-problem-inclusion} suppose that $A+C$ is $\rho$-strongly monotone and let $(\bx,\bv_1,\ldots,\bv_m) \in \h\times\fG$ be a primal-dual solution to Problem \ref{opt-problem-inclusion}. Then, for any $\varepsilon > 0$, there exists some $n_0\in\N$ (depending on $\varepsilon$ and $\rho\gamma_0$) such that for any $n\geq n_0$
 \begin{align}
	\label{ineq-th2-1}
	\| x_{n} - \bx \|^2 \leq \frac{1+\varepsilon}{n^2} \left( \frac{\| x_0 - \bx \|^2}{\rho^2\gamma_0^2} + \sum_{i=1}^m\frac{\| v_{i,0} - \bv_i \|^2}{\rho^2\gamma_0\sigma_0}\right),
\end{align}
where $\gamma_n,\,\sigma_n \in \R_{++}$, $x_n \in \h$ and $(v_{1,n},\ldots,v_{m,n}) \in \fG$ are the iterates generated by Algorithm \ref{alg2}.
\end{theorem}
\begin{remark}\label{rm2}
In Algorithm \ref{alg2} and Theorem \ref{theorem-inclusion2} we assumed that $D_i^{-1} = 0$ for $i=1,\ldots,m$, however, similar statements can be also provided for Problem \ref{opt-problem-inclusion-full} under the additional assumption that the operators $D_i : \g_i \rightarrow 2^{\g_i}$ are such that $D_i^{-1}$ is $\nu_i^{-1}$-cocoercive with $\nu_i \in \R_{++}$ for $i=1,\ldots,m$. This assumption is in general stronger than assuming that $D_i$ is monotone and $D_i^{-1}$ is $\nu_i$-Lipschitzian for $i=1,...,m$. However, it guarantees that $D_i$ is $\nu_i^{-1}$-strongly monotone  and maximally monotone for $i=1,...,m$ (see \cite[Example 20.28, Proposition 20.22 and Example 22.6]{BauschkeCombettes11}). We introduce the Hilbert space $\fH = \h \times \fG$, the element $\fz=(z,0,\ldots,0) \in \fH$ and the maximally monotone operator $\fA : \fH \rightarrow 2^{\fH}$, $\fA (x,y_1,\ldots,y_m) = (Ax,D_1y_1,\ldots,D_m y_m)$ and the monotone and Lipschitzian operator $\fC : \fH \rightarrow \fH$, $\fC (x,y_1,\ldots,y_m) = (Cx,0,\ldots,0)$. Notice also that $\fA + \fC$ is strongly monotone. Furthermore, we introduce the element $\fr=(r_1,\ldots,r_m)\in\fG$, the maximally monotone operator $\fB: \fG \rightarrow 2^{\fG}$, $\fB(y_1,\ldots,y_m) = (B_1y_1,\ldots,B_my_m)$, and the linear continuous operator $\fL : \fH \rightarrow \fG$, $\fL (x,y_1\ldots,y_m) = (L_1 x-y_1,\ldots,L_m x - y_m),$ having as adjoint $\fL^* : \fG \rightarrow \fH$, $\fL^* (q_1,\ldots,q_m) = (\sum_{i=1}^m L_i^* q_i, -q_1,\ldots,-q_m)$. We consider the primal problem
	\begin{align}
	\label{primal-inc-pspace}
		\text{find } \fbx = (\bx,\bp_1\ldots\bp_m) \in \fH \text{ such that } \fz \in \fA \fbx + \fL^*\fB \left(\fL \fbx - \fr \right) + \fC \fbx,
	\end{align}
together with the dual inclusion problem
\begin{align}
	\label{dual-inc-pspace}
	\text{find } \fbv \in \fG \text{ such that }(\exists \fx \in\fH)\left\{
	\begin{array}{l}
	\fz - \fL^* \fbv \in \fA \fx + \fC \fx \\
		\fbv \in \fB(\fL \fx -\fr )
	\end{array}
\right..
\end{align}	
We notice that Algorithm \ref{alg2} can be employed for solving this primal-dual pair of monotone inclusion problems and, by separately involving the resolvents of $A, B_i$ and $D_i, i=1,...,m$, as for $\gamma \in \R_{++}$
\begin{align*}
		J_{\gamma \fA}(x,y_1,\ldots,y_m) &= (J_{\gamma A}x, J_{\gamma D_1} y_1, \ldots, J_{\gamma D_m} y_m) \ \forall (x,y_1,\ldots,y_m) \in \fH \\
		J_{\gamma \fB}(q_1,\ldots,q_m) &= (J_{\gamma B_1} q_1, \ldots, J_{\gamma B_m} q_m) \ \forall (q_1,\ldots,q_m) \in \fG.
\end{align*}
Having $(\fbx,\fbv) \in \fH \times \fG$ a primal-dual solution to the primal-dual pair of monotone inclusion problems \eqref{primal-inc-pspace}-\eqref{dual-inc-pspace}, Algorithm \ref{alg2} generates a sequence of primal iterates fulfilling \eqref{ineq-th2-1} in $\fH$. Moreover, $(\fbx,\fbv)$ is a a primal-dual solution to \eqref{primal-inc-pspace}-\eqref{dual-inc-pspace} if and only if
$$\fz - \fL^*\fbv \in \fA \fbx + \fC \fbx \ \mbox{and} \  \fbv \in \fB \left(\fL \fbx - \fr \right)$$
$$\Leftrightarrow z - \sum_{i=1}^m L_i^*\bv_i \in A \bx + C \bx \ \mbox{and} \  \bv_i \in D_i \bp_i, \bv_i \in B_i \left(L_i \bx - \bp_i - r_i \right), i=1,\ldots,m$$
$$\Leftrightarrow z - \sum_{i=1}^m L_i^*\bv_i \in A \bx + C \bx \ \mbox{and} \  \bv_i \in D_i \bp_i, L_i \bx - r_i \in B_i^{-1}\bv_i + \bp_i, i=1,\ldots,m.$$
Thus, if $(\fbx,\fbv)$ is a primal-dual solution to \eqref{primal-inc-pspace}-\eqref{dual-inc-pspace}, then $(\bx,\fbv)$ is a primal-dual solution to Problem
\ref{opt-problem-inclusion-full}. Viceversa, if $(\bx,\fbv)$ is a primal-dual solution to Problem \ref{opt-problem-inclusion-full}, then, choosing $\bp_i \in D_i^{-1}\bv_i, i=1,...,m$, and $\fbx = (\bx,\bp_1\ldots\bp_m)$, it yields that $(\fbx,\fbv)$ is a primal-dual solution to \eqref{primal-inc-pspace}-\eqref{dual-inc-pspace}. In conclusion, the first component of every primal iterate in $\fH$ generated by Algorithm \ref{alg2} for finding the primal-dual solution $(\fbx,\fbv)$ to \eqref{primal-inc-pspace}-\eqref{dual-inc-pspace} will furnish a sequence of iterates verifying \eqref{ineq-th2-1} in $\h$ for the primal-dual solution $(\bx,\fbv)$ to Problem
\ref{opt-problem-inclusion-full}.
\end{remark}

\subsection{The case when \texorpdfstring{$A+C$ and $B_i^{-1}+D_i^{-1}$, $i=1,\ldots,m,$}{A+C and Bi(inverse)+Di(inverse), i=1,...,m} are strongly monotone}
Within this subsection we consider the case when $A+C$ is $\rho$-strongly monotone with $\rho \in \R_{++}$ and $B_i^{-1}+D_i^{-1}$ is $\tau_i$-strongly monotone with $\tau_i \in \R_{++}$ for $i=1,\ldots,m,$ and provide an accelerated version of the algorithm in Theorem \ref{theorem-inclusion-prelim} which generates sequences of primal and dual iterates that converge to the primal-dual solution to Problem \ref{opt-problem-inclusion-full} with an improved rate of convergence.
\begin{algorithm}\label{alg3}
	Let $x_0 \in \h$, $(v_{1,0}, \ldots, v_{m,0}) \in \fG$, and $\gamma \in (0,1)$ such that
	$$ \gamma \leq  \frac{1}{\sqrt{1+2\min\left\{\rho,\tau_1,\ldots,\tau_m\right\}}\left(\sqrt{\sum_{i=1}^m \| L_i\|^2} + \max \left\{ \mu, \nu_1,\ldots,\nu_m \right\}\right)}. $$
	Consider the following updates:
	\begin{align}\label{A3}
	  \left(\forall n\geq 0\right) \   \left\lfloor \begin{array}{l}
		p_{1,n} = J_{\gamma A}\left(x_n - \gamma \left( C x_n +\sum_{i=1}^mL_i^*v_{i,n} -z\right) \right) \\
		\text{For }i=1,\ldots,m  \\
				\ \left\lfloor \begin{array}{l}
					p_{2,i,n} = J_{\gamma B_i^{-1}}\left(v_{i,n} +\gamma (L_i x_n - D_i^{-1} v_{i,n} -r_i)\right) \\
					v_{i,n+1} = \gamma L_i( p_{1,n} - x_n) +\gamma (D_i^{-1} v_{i,n} -D_i^{-1} p_{2,i,n}) + p_{2,i,n}
				\end{array} \right.\\
		x_{n+1} = \gamma \sum_{i=1}^mL_i^*(v_{i,n}-p_{2,i,n}) + \gamma (C x_n - C p_{1,n}) +p_{1,n}.		
		\end{array}
		\right.
	\end{align}
\end{algorithm}

\begin{theorem}\label{theorem-inclusion-strongly2}
In Problem \ref{opt-problem-inclusion-full} suppose that $A+C$ is $\rho$-strongly monotone with $\rho \in \R_{++}$, $B_i^{-1}+D_i^{-1}$ is $\tau_i$-strongly monotone with $\tau_i \in \R_{++}$ for $i=1,\ldots,m,$ and let $(\bx,\bv_1,\ldots,\bv_m) \in \h\times\fG$ be a primal-dual solution to Problem \ref{opt-problem-inclusion-full}. Then for every $n\geq 0$ it holds
\begin{align*}
	\| x_{n} - \bx \|^2 + \sum_{i=1}^m \| v_{i,n} - \bv_i \|^2 \leq  \left(\frac{1}{1+2\rho_{\min}\gamma(1-\gamma)}\right)^n \left( \| x_{0} - \bx \|^2 + \sum_{i=1}^m \| v_{i,0} - \bv_i \|^2\right),
\end{align*}
where $\rho_{\min}=\min\left\{\rho,\tau_1,\ldots,\tau_m\right\}$ and $x_n \in \h$ and $(v_{1,n},\ldots,v_{m,n}) \in \fG$ are the iterates generated by Algorithm \ref{alg3}.
\end{theorem}

\begin{proof}
Taking into account the definitions of the resolvents occurring in Algorithm \ref{alg3} we obtain for every $n \geq 0$
\begin{align*}
	\frac{x_n-x_{n+1}}{\gamma} - \sum_{i=1}^m L_i^*p_{2,i,n} +z &\in (A+C)p_{1,n}
\end{align*}
and
\begin{align*}
	\frac{v_{i,n}-v_{i,n+1}}{\gamma} + L_i p_{1,n} - r_i &\in (B_i^{-1} + D_i^{-1}) p_{2,i,n}, \ i=1,\ldots,m.
\end{align*}
The primal-dual solution $(\bx,\bv_1,\ldots,\bv_m) \in \h\times\fG$ to Problem \ref{opt-problem-inclusion-full} fulfills (see \eqref{operator-proof-conditions-full})
\begin{equation*}
z - \sum_{i=1}^m L_i^*\bv_i \in A\bx + C\bx \ \mbox{and} \  \bv_i \in (B_i \Box D_i)(L_i\bx-r_i), \,i=1,\ldots,m.
\end{equation*}
By the strong monotonicity of $A+C$ and $B_i^{-1}+D_i^{-1}$, $i=1,\ldots,m$, we obtain for every $n \geq 0$
\begin{align}\label{ineq-sm1}
	\< p_{1,n} - \bx, \frac{x_n-x_{n+1}}{\gamma} - \sum_{i=1}^m L_i^*p_{2,i,n} +z - \left( z - \sum_{i=1}^m L_i^*\bv_i \right)  \> &\geq \rho \|  p_{1,n} - \bx\|^2
\end{align}
and, respectively,
\begin{align}\label{ineq-sm2}
	\< p_{2,i,n} - \bv_i, \frac{v_{i,n}-v_{i,n+1}}{\gamma} + L_i p_{1,n} - r_i - \left( L_i \bx -r_i \right)\> &\geq \tau_i \| p_{2,i,n} -\bv_i \|^2, i=1,...,m.
 \end{align}
Consider the Hilbert space $ \fH = \h \times \fG$, equipped with the inner product defined in \eqref{inprodfH} and associated norm, and set
$$ \fbx = (\bx,\bv_1,\ldots,\bv_m), \quad \fx_n=(x_n,v_{1,n},\ldots,v_{m,n}), \quad \fp_n=(p_{1,n},p_{2,1,n},\ldots,p_{2,m,n}).$$
Summing up the inequalities \eqref{ineq-sm1} and \eqref{ineq-sm2} and using
\begin{align*}
	\< \fp_n - \fbx , \frac{\fx_n - \fx_{n+1}}{\gamma} \> = \frac{\| \fx_{n+1} - \fp_{n} \|^2}{2\gamma} - \frac{\|\fx_n- \fp_{n}\|^2}{2\gamma} + \frac{\| \fx_n - \fbx \|^2}{2\gamma} - \frac{\| \fx_{n+1} - \fbx \|^2}{2\gamma},
\end{align*}
we obtain for every $n \geq 0$
\begin{align}
	\label{ineq-sm3}
	\frac{\| \fx_n - \fbx \|^2}{2\gamma} \geq \rho_{\min} \|  \fp_{n} - \fbx\|^2 + \frac{\| \fx_{n+1} - \fbx \|^2}{2\gamma} + \frac{\|\fx_n- \fp_{n}\|^2}{2\gamma} - \frac{\| \fx_{n+1} - \fp_{n} \|^2}{2\gamma}.
\end{align}
Further, using the estimate $2ab\leq \gamma a^2 + \frac{b^2}{\gamma}$ for all $a,b \in \R$,  we obtain
\begin{align*}
	\rho_{\min} \|  \fp_{n} - \fbx\|^2 &\geq \frac{2\rho_{\min}\gamma(1-\gamma)}{2\gamma} \|  \fx_{n+1} - \fbx\|^2 - \frac{2\rho_{\min}(1-\gamma)}{2\gamma} \|  \fx_{n+1} - \fp_{n}\|^2 \\
	&\geq \frac{2\rho_{\min}\gamma(1-\gamma)}{2\gamma} \|  \fx_{n+1} - \fbx\|^2 - \frac{2\rho_{\min}}{2\gamma} \|  \fx_{n+1} - \fp_{n}\|^2 \ \forall n \geq 0.
\end{align*}
Hence, \eqref{ineq-sm3} reduces to
\begin{align*}
	\frac{\| \fx_n - \fbx \|^2}{2\gamma} &\geq \frac{ (1+2\rho_{\min}\gamma(1-\gamma) )\| \fx_{n+1} - \fbx \|^2}{2\gamma} \\
	&\quad + \frac{\|\fx_n- \fp_{n}\|^2}{2\gamma} - \frac{(1+2\rho_{\min}) \| \fx_{n+1} - \fp_{n} \|^2}{2\gamma} \ \forall n \geq 0.
\end{align*}
Using the same arguments as in \eqref{ineq-lipschitz-continuity}, it is easy to check that for every $n \geq 0$
\begin{align*}
	&\frac{\|\fx_n- \fp_{n}\|^2}{2\gamma} - \frac{(1+2\rho_{\min}) \| \fx_{n+1} - \fp_{n} \|^2}{2\gamma} \\
	&\geq \left(1-(1+2\rho_{\min})\gamma^2\left(\sqrt{\sum_{i=1}^m \| L_i\|^2} + \max \left\{ \mu, \nu_1,\ldots,\nu_m \right\}\right)^2\right) \frac{\|\fx_n- \fp_{n}\|^2}{2\gamma}\\
&\geq 0,
\end{align*}
whereby the nonnegativity of this term is ensured by the assumption that
$$ \gamma \leq \frac{1}{\sqrt{1+2\rho_{\min}}\left(\sqrt{\sum_{i=1}^m \| L_i\|^2} + \max \left\{ \mu, \nu_1,\ldots,\nu_m \right\}\right)}. $$
Therefore, we obtain
\begin{align*}
	\| \fx_n - \fbx \|^2 \geq (1+2\rho_{\min}\gamma(1-\gamma) )\| \fx_{n+1} - \fbx \|^2 \ \forall n \geq 0,
\end{align*}
which leads to
 \begin{align*}
	\| \fx_{n} - \fbx \|^2 \leq  \left(\frac{1}{1+2\rho_{\min}\gamma(1-\gamma)}\right)^n \| \fx_0 - \fbx \|^2 \ \forall n \geq 0.
\end{align*}
\end{proof}

\section{Numerical experiments in imaging}\label{sectionApp}
In this section we test the feasibility of Algorithm \ref{alg1} and of its accelerated version Algorithm \ref{alg2} in the context of different problem formulations occurring in imaging and compare their performances to the ones of two other popular primal-dual algorithms introduced in \cite{ChaPoc11}. For all applications discussed in this section the images have been normalized, in order to make their pixels range in the closed interval from $0$ (pure black) to $1$ (pure white).

\subsection{TV-based image denoising}\label{subsectionDenoise}
Our first numerical experiment aims the solving of an image denoising problem via total variation regularization. More precisely, we deal with the convex optimization problem
\begin{align}
	\label{ex-denoise}
	\inf_{x\in\R^n} \left\{ \lambda \, TV(x) + \frac{1}{2} \|x-b\|^2 \right\},
\end{align}
where $\lambda\in\R_{++}$ is the regularization parameter, $TV:\R^n \rightarrow \R$ is a discrete total variation functional and $b\in\R^n$ is the observed noisy image.

In this context, $x \in \R^n$ represents the vectorized image $X\in\R^{M\times N}$, where $n = M\cdot N$ and $x_{i,j}$ denotes the normalized value of the pixel located in the $i$-th row and the $j$-th column, for $i=1,\ldots,M$ and $j=1,\ldots,N$. Two popular choices for the discrete total variation functional are the \textit{isotropic total variation} $TV_{\iso}:\R^n \rightarrow \R$,
\begin{align*}
	TV_{\iso}(x) &= \sum_{i=1}^{M-1}\sum_{j=1}^{N-1}\sqrt{ (x_{i+1,j}-x_{i,j})^2 + (x_{i,j+1}-x_{i,j})^2 } \\
								&\quad + \sum_{i=1}^{M-1} \left| x_{i+1,N}-x_{i,N} \right|  + \sum_{j=1}^{N-1} \left| x_{M,j+1}-x_{M,j} \right| ,
\end{align*}
and the \textit{anisotropic total variation} $TV_{\aniso}:\R^n \rightarrow \R$,
\begin{align*}
	TV_{\aniso}(x) &= \sum_{i=1}^{M-1}\sum_{j=1}^{N-1} \left|x_{i+1,j}-x_{i,j}\right| + \left|x_{i,j+1}-x_{i,j}\right| \\
								&\quad  + \sum_{i=1}^{M-1} \left| x_{i+1,N}-x_{i,N} \right| + \sum_{j=1}^{N-1} \left| x_{M,j+1}-x_{M,j} \right| ,
\end{align*}
where in both cases reflexive (Neumann) boundary conditions are assumed.

We denote $\Y=\R^n \times \R^n$ and define the linear operator $L:\R^n \rightarrow \Y$, $x_{i,j} \mapsto (L_1x_{i,j}, L_2x_{i,j})$, where
\begin{align*}
	L_1x_{i,j} = \left\{ \begin{array}{ll} x_{i+1,j}-x_{i,j}, & \text{if }i<M\\ 0, &\text{if }i=M\end{array}\right. \ \mbox{and} \
	L_2x_{i,j} = \left\{ \begin{array}{ll} x_{i,j+1}-x_{i,j}, & \text{if }j<N\\ 0, &\text{if }j=N\end{array}\right. .
\end{align*}
The operator $L$ represents a discretization of the gradient using reflexive (Neumann) boundary conditions and standard finite differences. One can easily check that $\| L \|^2 \leq 8$ and that its adjoint $L^* : \Y \rightarrow \R^n$ is as easy to implement as the operator itself (cf. \cite{Cha04}).

Within this example we will focus on the anisotropic total variation function which is nothing else than the composition of the $l_1$-norm in $\Y$ with the linear operator $L$. Due to the full splitting characteristics of the iterative methods presented in this paper, we need only to compute
the proximal point of the conjugate of the $l_1$-norm, the latter being the indicator function of the dual unit ball. Thus, the calculation of the proximal point will result in the computation of a projection, which has an easy implementation. The more challenging isotropic total variation functional is employed in the forthcoming subsection in the context of an image deblurring problem.

Thus, problem \eqref{ex-denoise} reads equivalently
$$ \inf_{x\in\R^n} \left\{ h(x) + g(Lx) \right\},$$
where $h: \R^n \rightarrow \R$, $h(x)=\frac{1}{2} \|x-b\|^2$, is $1$-strongly monotone and differentiable with $1$-Lipschitzian gradient and $g:\Y \rightarrow \R$ is defined as $g(y_1,y_2)=\lambda \|(y_1,y_2)\|_1$. Then its conjugate $g^*:\Y \rightarrow \oR$ is nothing else than
$$g^*(p_1,p_2) = \left(\lambda\|\cdot\|_1\right)^*(p_1,p_2)= \lambda \left\|\left(\frac{p_1}{\lambda}, \frac{p_2}{\lambda}\right)\right\|_1^* = \delta_{S}(p_1,p_2),$$
where $S=\left[-\lambda, \lambda\right]^n \times \left[-\lambda, \lambda\right]^n$. Taking $x_0 \in \h$, $v_0 \in \Y$,
$$	\gamma_0 \in  \left(0,\min\left\{1, \frac{\sqrt{1+4\rho}}{2(1+2\rho)\mu}\right\}\right)  \text{ and }
	  \sigma_0 =  \frac{1}{2\gamma_0(1+2\rho)\sum_{i=1}^m \|L_i \|^2},$$
Algorithm \ref{alg2} looks for this particular problem like
	\begin{align*}
	  \left(\forall n\geq 0\right) \   \left\lfloor \begin{array}{l}
		p_{1,n} = x_n - \gamma_n \left( x_n -b +L^* v_{n} \right) \\
		p_{2,n} = \proj_{S}\left(v_{n} +\sigma_n L x_n \right) \\
		v_{n+1} = \sigma_n L( p_{1,n} - x_n)  + p_{2,n} \\
		x_{n+1} = \gamma_n L^*(v_{n}-p_{2,n}) + \gamma_n(x_n - p_{1,n}) +p_{1,n} \\
		\theta_n=1/\sqrt{1+2\rho\gamma_n(1-\gamma_n)}, \ \gamma_{n+1} = \theta_n\gamma_n, \ \sigma_{n+1} = \sigma_n/\theta_n.
		\end{array}
		\right.
	\end{align*}
\begin{table}[tb]
	\centering
		\begin{tabular}{| l || l | l c | c l | l  | } \hline
		  & \multicolumn{2}{c}{$\sigma=0.12$, $\lambda=0.07$} & \hspace{0.1cm} & \hspace{0.1cm} & \multicolumn{2}{c|}{$\sigma=0.06$, $\lambda=0.035$} \\\cline{2-4}\cline{5-7}
		     & $\varepsilon=10^{-4}$ & $\varepsilon=10^{-6}$ &  &   &  $\varepsilon=10^{-4}$ & $\varepsilon=10^{-6}$ \\ \hline \hline
		ALG1 &  $350\ (7.03 \text{ s})$  & $2989\ (59.82 \text{ s})$  & & & $184\ (3.69 \text{ s})$  & $1454\ (29.07 \text{ s})$ \\
		ALG2 &  $101\ (2.28 \text{ s})$  &  $442\ (9.91 \text{ s})$     & & & $72\ (1.62 \text{ s})$ & $298\ (6.68 \text{ s})$\\
		PD1  & 	$342\ (3.59 \text{ s})$  &  $3133\ (32.68 \text{ s})$  & &	& $180\ (1.91 \text{ s})$	& $1427\ (14.87 \text{ s})$ \\
		PD2  &  $96\ (1.02 \text{ s})$  & $442\ (4.67 \text{ s})$      & & & $69\ (0.76 \text{ s})$ & $319\ (3.39 \text{ s})$ \\ \hline
	\end{tabular}
	\caption{\small Performance evaluation for the images in Figure \ref{fig:lichtenstein}. The entries represent to the number of iterations and the CPU times in seconds, respectively, needed in order to attain a root mean squared error for the iterates below the tolerance $\varepsilon$.}
	\label{table:performance}
\end{table}
However, we solved the regularized image denoising problem with Algorithm \ref{alg1}, the primal-dual iterative scheme from \cite{ChaPoc11} (see, also, \cite{Vu11}) and the accelerated version of the latter presented in \cite[Theorem 2]{ChaPoc11}, as well, and refer the reader to Table \ref{table:performance} for a comparison of the obtained results:
\begin{itemize}
\setlength{\itemsep}{-5pt}
	\item ALG1: Algorithm \ref{alg1} with $\gamma = \frac{1-\tilde\varepsilon}{\sqrt{8}}$, small $\tilde\varepsilon>0$ and by taking the last iterate instead of the averaged sequence.
	\item ALG2: Algorithm \ref{alg2} with $\rho=0.3$, $\mu=1$ and $\gamma_0=\frac{\sqrt{1+4\rho}}{2(1+2\rho)\mu}$.
	\item PD1: Algorithm 1 in \cite{ChaPoc11} with $\tau=\frac{1}{\sqrt{8}}$, $\tau \sigma 8 = 1$ and by taking the last iterate instead of the averaged sequence.
	\item PD2: Algorithm 2 in \cite{ChaPoc11} with $\rho=0.3$, $\tau_0=\frac{1}{\sqrt{8}}$, $\tau_0\sigma_0 8=1$.
\end{itemize}

From the point of view of the number of iterations, one can notice similarities between both the primal-dual algorithms ALG1 and PD1 and the accelerated versions ALG2 and PD2. From this point of view they behave almost equal. When comparing the CPU times, it shows that the methods in this paper need almost twice amount of time. This is since ALG1 and ALG2 lead back to a forward-backward-forward splitting, whereas PD1 and PD2 rely on a forward-backward splitting scheme, meaning that ALG1 and ALG2 process the double amount of forward steps than PD1 and PD2. In this example the evaluation of forward steps (i.\,e. which constitute in matix-vector multiplications involving the linear operators and their adjoints) is, compared with the calculation of projections when computing the resolvents, the most costly step.

\begin{figure}[tb]
	\floatbox[{\capbeside\thisfloatsetup{capbesideposition={right,top},capbesidewidth=4.2cm}}]{figure}[\FBwidth]
	{\captionsetup[subfigure]{position=top}
	\vspace{-0.3cm}
	\subfloat[Noisy image, $\sigma=0.06$]{\includegraphics*[viewport= 144 250 467 574, width=0.32\textwidth]{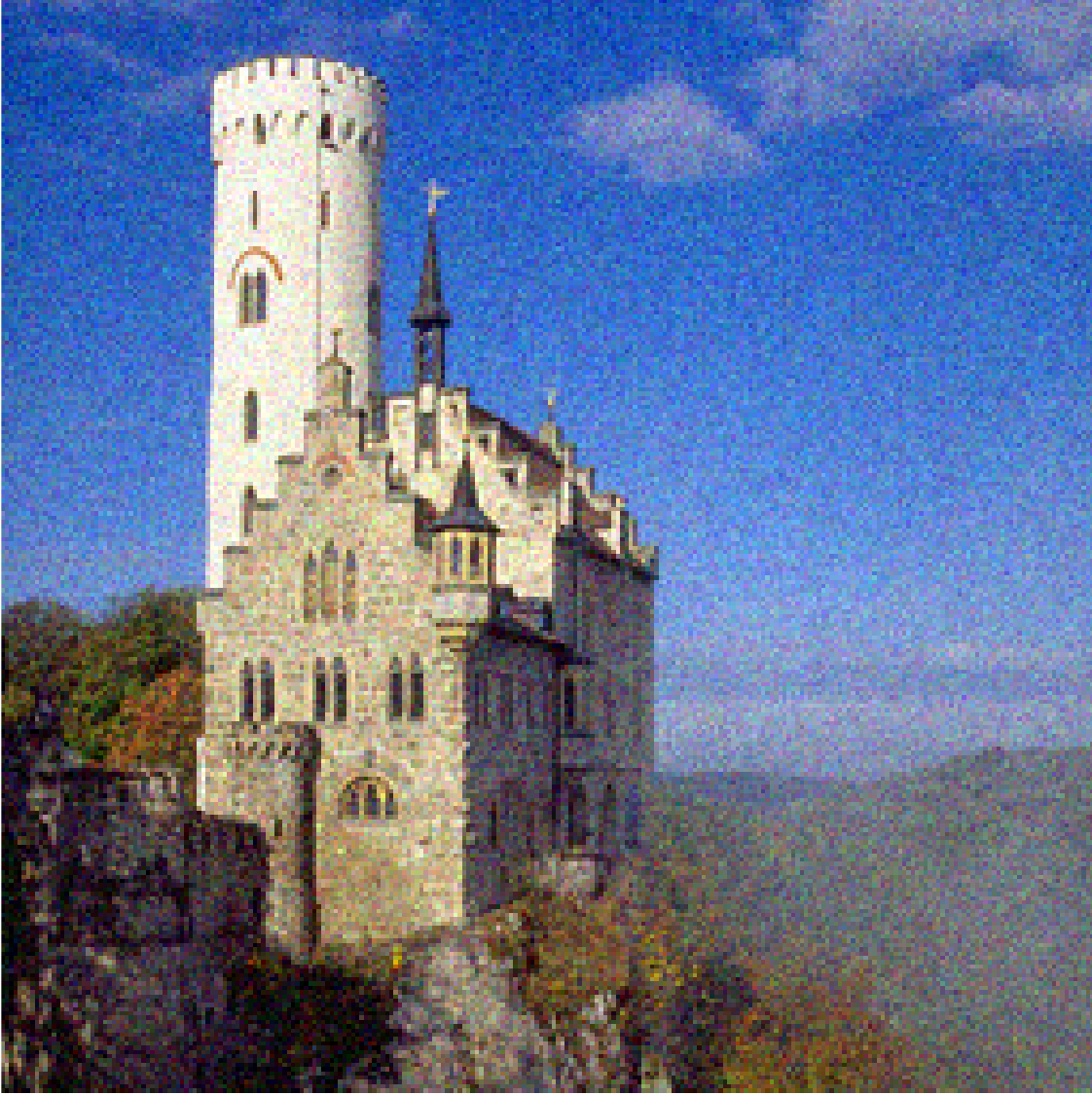}}  \hspace{0.2mm}
	\subfloat[Noisy image, $\sigma=0.12$]{\includegraphics*[viewport= 144 250 467 574, width=0.32\textwidth]{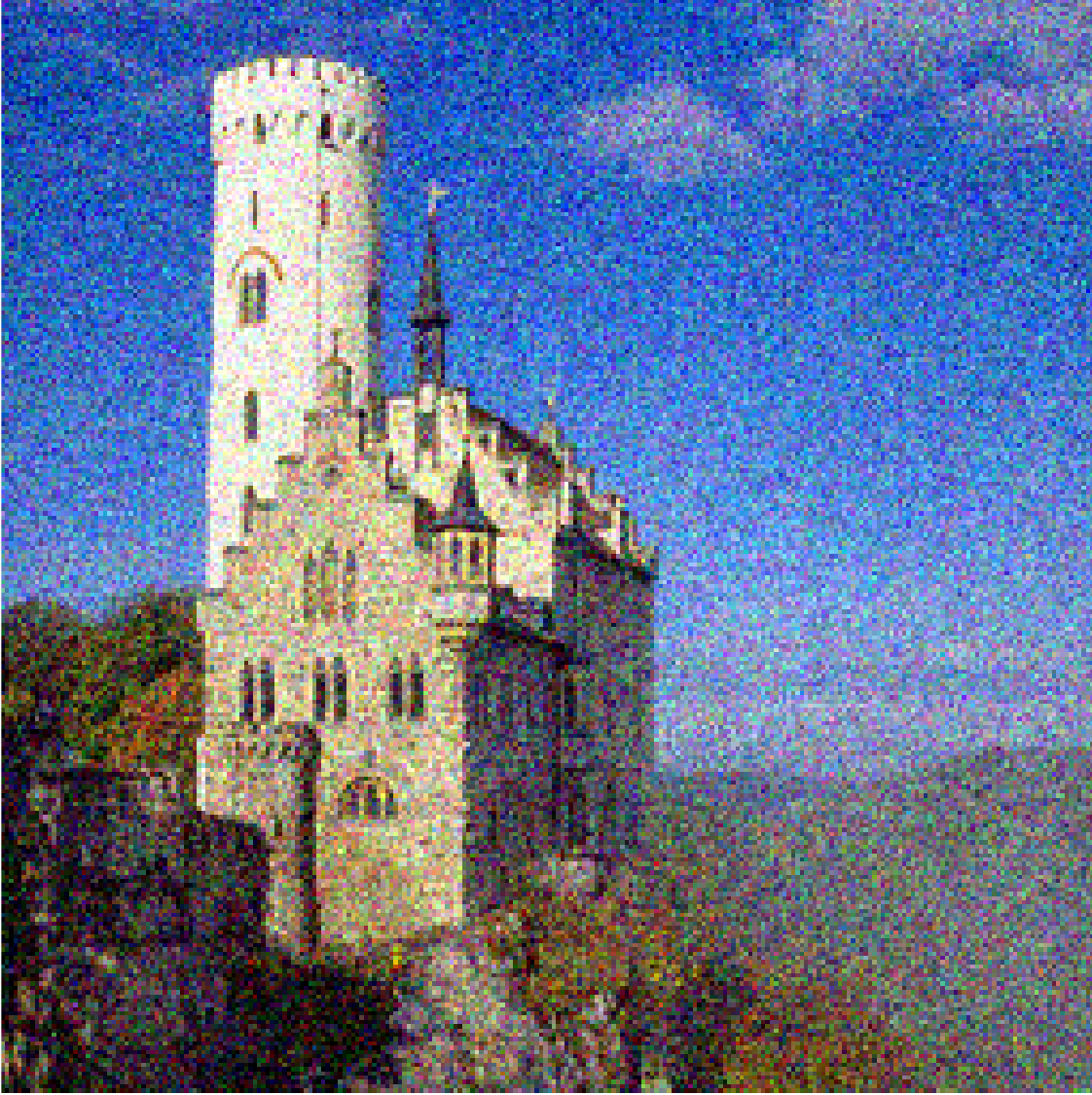}}
	
	\subfloat[Denoised image, $\lambda=0.035$]{\includegraphics*[viewport= 144 250 467 574, width=0.32\textwidth]{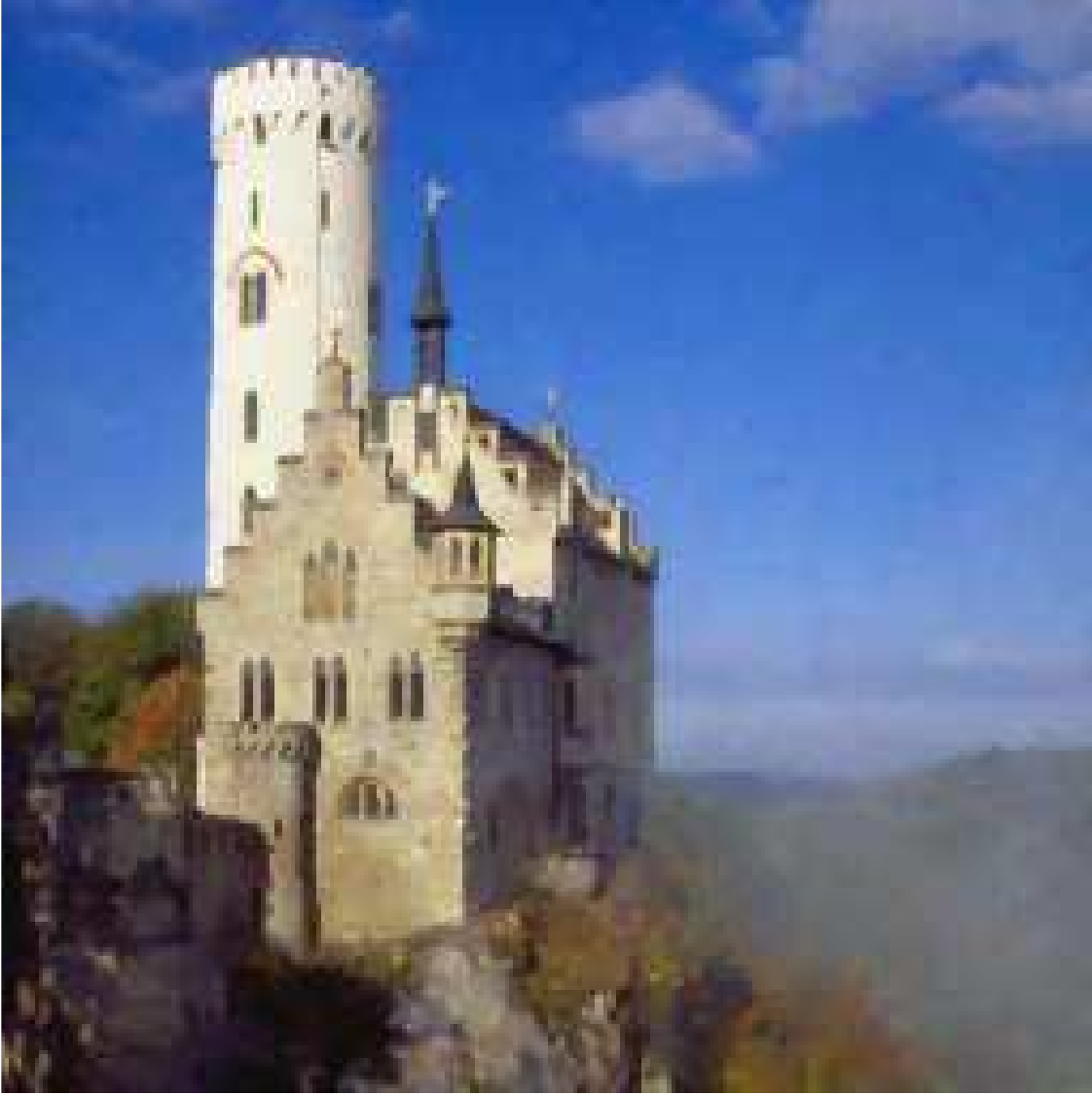}}	\hspace{0.2mm}
	\subfloat[Denoised image, $\lambda=0.07$]{\includegraphics*[viewport= 144 250 467 574, width=0.32\textwidth]{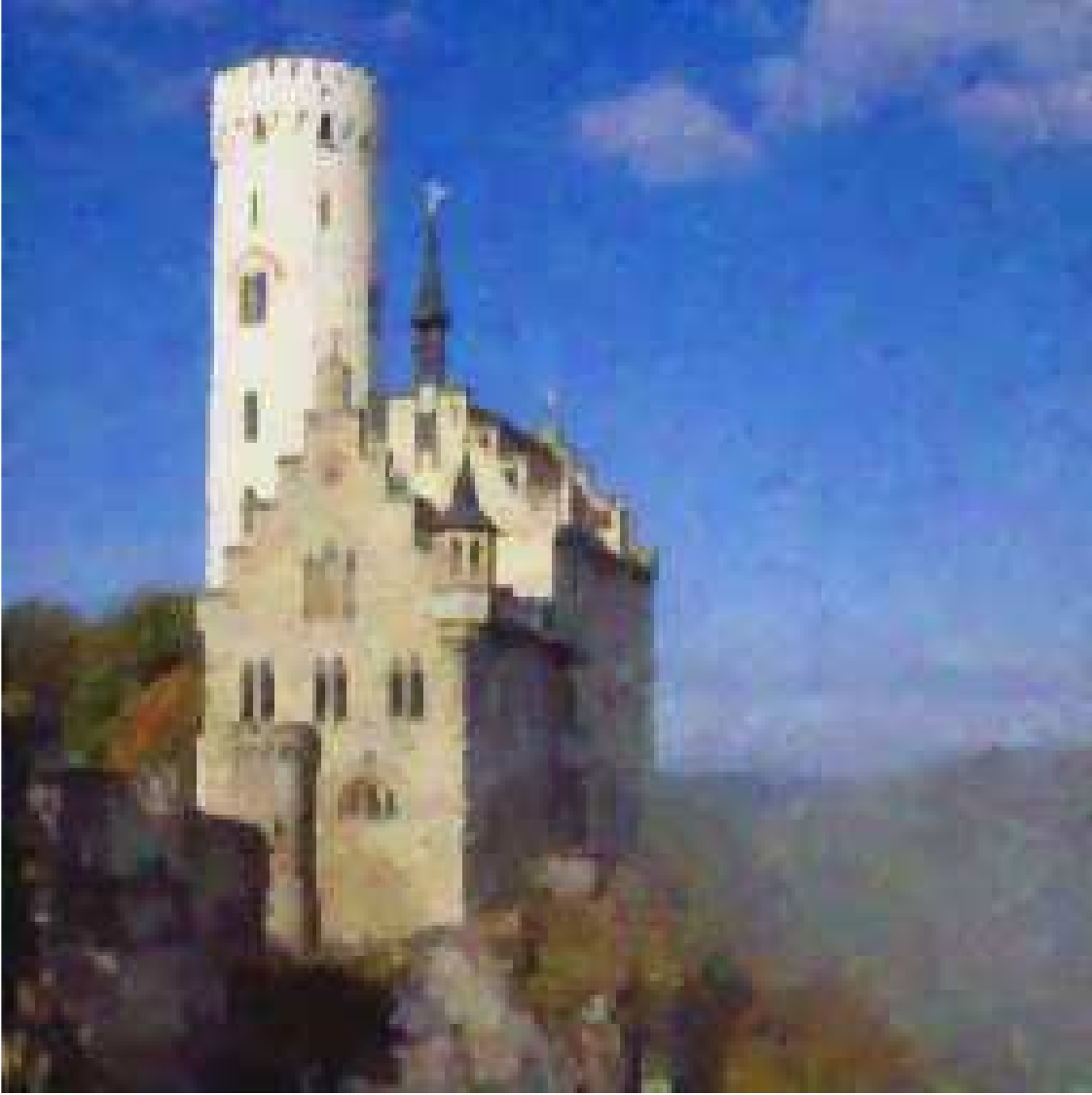}}	}
	{\hspace{-9cm}\caption{\small $TV$-$l_2$ image denoising. The noisy image in (a) was obtained after adding white Gaussian noise with standard deviation $\sigma=0.06$ to the original $256 \times 256$ lichtenstein test image, while the output of Algorithm \ref{alg2}, for  $\lambda=0.035$, after 100 iterations is shown in (c). Likewise, the noise image when choosing $\sigma=0.12$ and the output of the same algorithm, for  $\lambda=0.07$, after 100 iterations are shown in (b) and (d), respectively.}\label{fig:lichtenstein}}
\end{figure}

\subsection{\texorpdfstring{$TV$}{TV}-based image deblurring}\label{subsectionDeblur}

The second numerical experiment that we consider concerns the solving of an extremely ill-conditioned linear inverse problem which arises in image deblurring and denoising. For a given matrix $A \in \mathbb{R}^{n \times n}$ describing a blur operator and a given vector $b \in \R^n$ representing the blurred and noisy image, the task is to estimate the unknown original image $\bx\in\R^n$ fulfilling
$$A\bx=b.$$
To this end we basically solve the following regularized convex nondifferentiable problem
 \begin{equation}\label{probimageproc}
\inf_{x \in \R^n}{\left\{ \left\| Ax-b \right\|_1 +\lambda_1 TV_{\iso}(x) + \lambda_2 \left\| x \right\|_1 + \delta_{\left[0,1\right]^n}(x) \right\}},
\end{equation}
where $\lambda_1,\,\lambda_2 \in \R_{++}$ are regularization parameters and $TV_{\iso}:\R^n \rightarrow \R$ is the discrete isotropic total variation function. Notice that none of the functions occurring in \eqref{probimageproc} is differentiable, while the regularization is done by a combination of two regularization functionals with different properties.

The blurring operator is constructed by making use of the Matlab routines {\ttfamily imfilter} and {\ttfamily fspecial} as follows:
\lstset{language=Matlab}
\begin{lstlisting}[numbers=left,numberstyle=\tiny,frame=tlrb,showstringspaces=false]
H=fspecial('gaussian',9,4); % gaussian blur of size 9 times 9
                            % and standard deviation 4
B=imfilter(X,H,'conv','symmetric'); % B=observed blurred image
                                    % X=original image
\end{lstlisting}
The function {\ttfamily fspecial} returns a rotationally symmetric Gaussian lowpass filter of size $9 \times 9$ with standard deviation $4$,  the entries of $H$ being nonnegative and their sum adding up to $1$. The function {\ttfamily imfilter} convolves the filter $H$  with the image $X$ and furnishes the blurred image $B$. The boundary option ``symmetric'' corresponds to reflexive boundary conditions. Thanks to the rotationally symmetric filter $H$, the linear operator $A$ defined via the routine {\ttfamily imfilter} is symmetric, too. By making use of the real spectral decomposition of $A$, it shows that $\left\| A \right\|^2=1$.

For $(y,z),\,(p,q) \in \Y$, we introduce the inner product $$\< (y,z),(p,q) \> = \sum_{i=1}^M\sum_{j=1}^N y_{i,j}p_{i,j} + z_{i,j}q_{i,j}$$ and define $\| (y,z)\|_{\times} = \sum_{i=1}^M\sum_{j=1}^N \sqrt{y_{i,j}^2 + z_{i,j}^2}$. One can check that $\|\cdot\|_{\times}$ is a norm on $\Y$ and that for every $x\in\R^n$ it holds $TV_{\iso}(x) = \| L x \|_{\times}$, where $L$ is the linear operator defined in the previous section. The conjugate function $(\|\cdot\|_{\times})^*:\Y \rightarrow \oR$ of $\|\cdot\|_{\times}$ is for every $(p,q) \in \Y$ given by (see, for instance, \cite{BotGradWanka09})
$$ (\|\cdot\|_{\times})^* (p,q) = \left\{ \begin{array}{ll}0, & \text{if }\|(p,q)\|_{\times *} \leq 1 \\ +\infty, & \text{otherwise} \end{array}\right.,$$
where
$$\|(p,q)\|_{\times *} = \sup_{\|(y,z)\|_{\times} \leq 1} \< (p,q),(y,z) \> = \max_{\substack{1 \leq i \leq M \\ 1 \leq j \leq N}}\sqrt{p_{i,j}^2 + q_{i,j}^2}. $$
Therefore, the optimization problem \eqref{probimageproc} can be written in the form of
\begin{equation*}
\inf_{x \in \R^n}{\left\{f(x) + g_1(Ax) + g_2(Lx)\right\}},
\end{equation*}
where $f: \R^n \rightarrow \oR$, $f(x)=\lambda_2 \|x\|_1 + \delta_{\left[0,1\right]^n}(x)$, $g_1:\R^n \rightarrow \R$, $g_1(y)=\left\| y-b \right\|_1$ and $g_2:\Y \rightarrow \R$, $g_2(y,z)=\lambda_1 \left\| (y,z) \right\|_{\times}$. For every $p \in \R^n$ it holds $g_1^*(p)=\delta_{\left[ -1, 1\right]^n}(p)+p^Tb$ (see, for instance, \cite{Bot10}), while for any $(p,q)\in\Y$ we have $g_2^*(p,q)=\delta_{S}(p,q)$, with $S=\{(p,q) \in \Y : \|(p,q)\|_{\times *} \leq \lambda_1\}$.
\begin{figure}[tb]	
	\centering
	\captionsetup[subfigure]{position=top}
	\subfloat[Original image]{\includegraphics*[viewport= 144 250 467 574, width=0.32\textwidth]{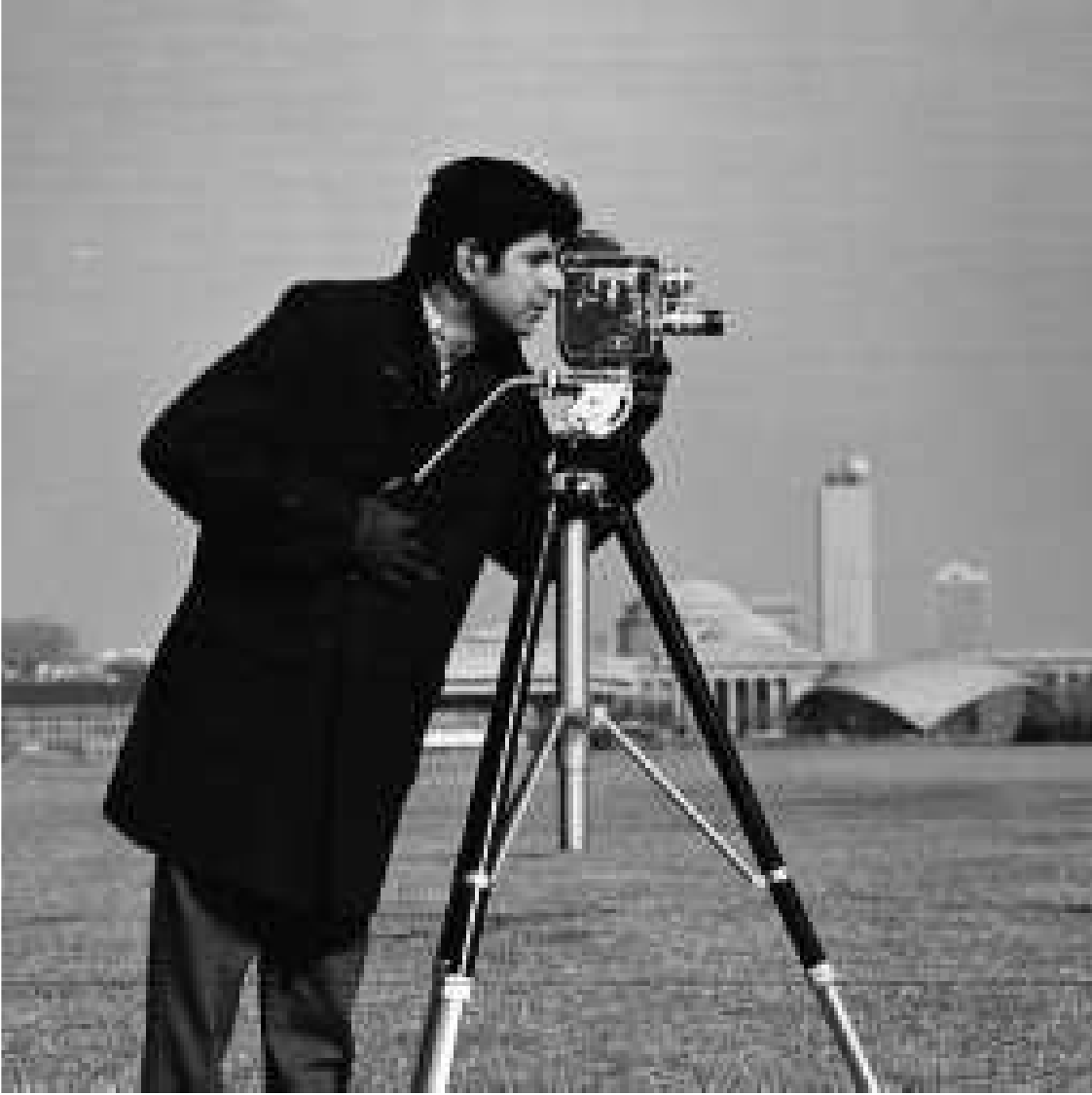}} \hspace{0.2mm}
	\subfloat[Blurred and noisy image]{\includegraphics*[viewport= 144 250 467 574, width=0.32\textwidth]{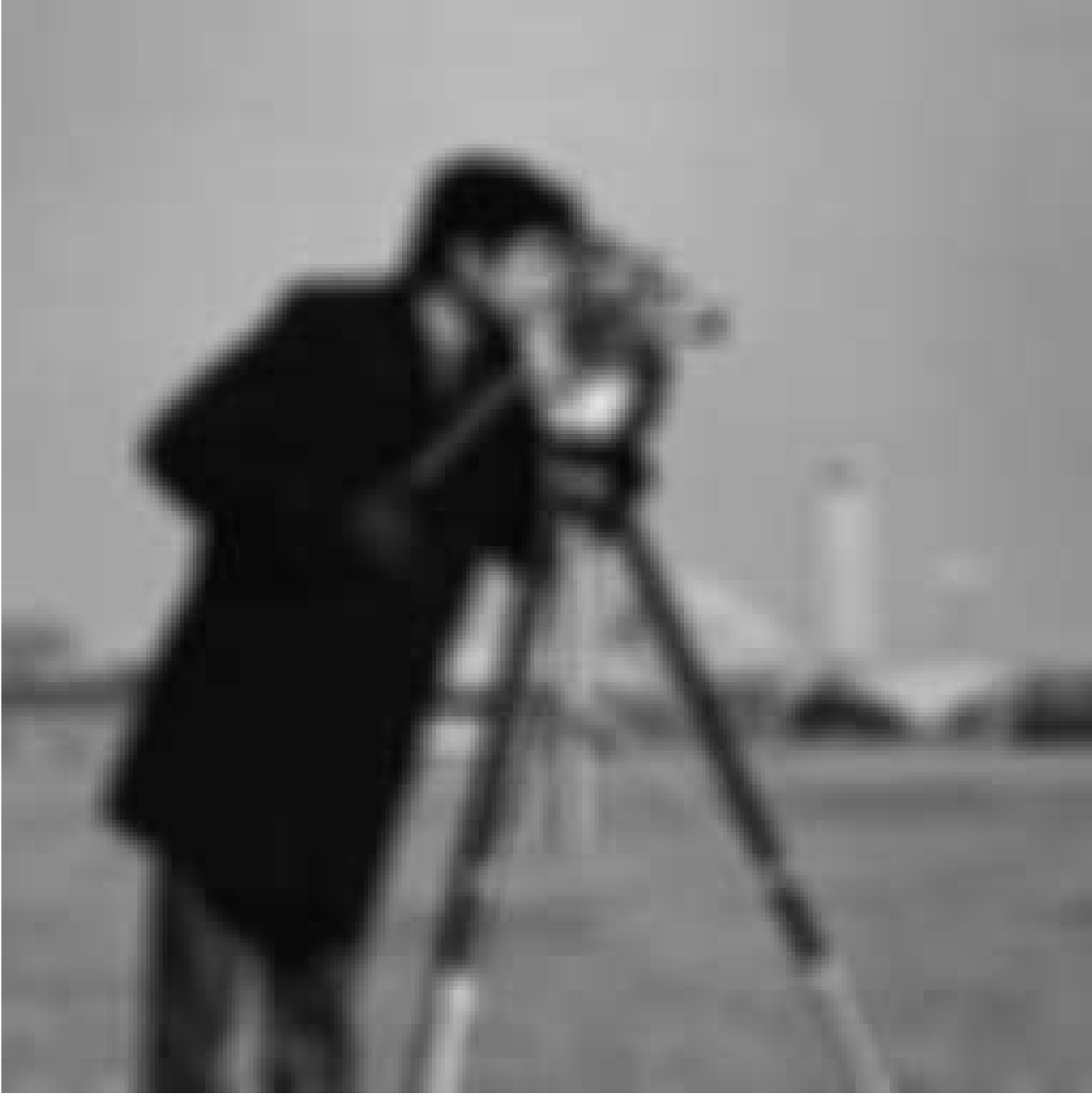}}  \hspace{0.2mm}
	\subfloat[Reconstructed image]{\includegraphics*[viewport= 144 250 467 574, width=0.32\textwidth]{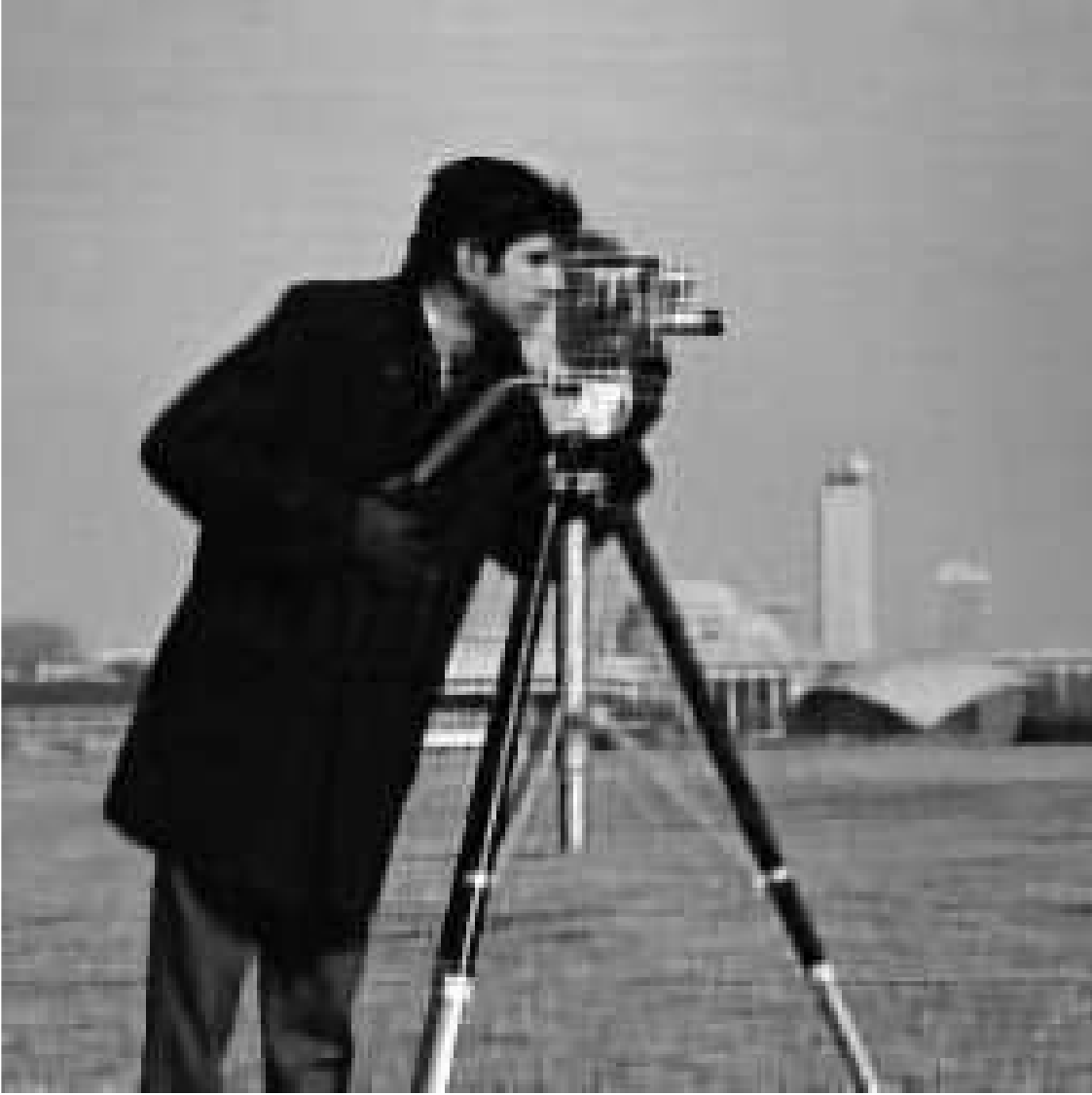}}	
	\caption{\small $TV$-$l_1$-$l_1$ image deblurring. Figure (a) shows the clean $256\times 256$ cameraman test image, (b) shows the image obtained after multiplying it with a blur operator and adding white Gaussian noise and (c) shows the averaged sequence generated by Algorithm  \ref{alg1} after 400 iterations.}
	\label{fig:cameraman}	
\end{figure}
We solved this problem by Algorithm \ref{alg1} and to this end we made use of the following formulae for the proximal points involved in the formulation of this iterative scheme:
\begin{align*}
	\Prox_{\gamma f}(x)&=\argmin_{z\in\left[0,1\right]^n}\left\{\gamma\lambda_2\|z\|_1 + \frac{1}{2}\|z-x\|^2\right\} = \proj_{\left[0,1\right]^n} \left(x-\gamma \lambda_2 \mathbbm{1}^n\right) \ \forall x\in\R^n,\\
	\Prox_{\gamma g_1^*}(p) &= \argmin_{z\in\left[-1,1\right]^n}\left\{ \gamma \< z,b\> + \frac{1}{2}\| z-p \|^2 \right\} = \proj_{\left[-1,1\right]^n} \left(p-\gamma b\right) \ \forall p\in\R^n,
\end{align*}
and
\begin{align*}
	\Prox_{\gamma g_2^*}(p,q)&=\argmin_{(y,z)\in S} \frac{1}{2}\|(y,z)-(p,q)\|^2 = \proj_{S} \left(p,q\right) \ \forall (p,q)\in\Y,
\end{align*}
where $\gamma \in \R_{++}$, $\mathbbm{1}^n$ is the vector in $\R^n$ with all entries equal to $1$ and the projection operator $\proj_S:\Y \rightarrow S$ is defined as
$$(p_{i,j},q_{i,j}) \mapsto \left(\frac{p_{i,j}}{\max\left\{1,\frac{\sqrt{p_{i,j}^2+q_{i,j}^2}}{\lambda_1}\right\}}, \frac{q_{i,j}}{\max\left\{1,\frac{\sqrt{p_{i,j}^2+q_{i,j}^2}}{\lambda_1}\right\}}\right).$$

Taking $x_0 \in \h$, $(v_{1,0},v_{2,0}) \in \R^n \times \Y$, $ \beta = \sqrt{1 + 8}=3$, $\varepsilon \in \left(0, \frac{1}{\beta +1} \right)$ and $(\gamma_n)_{n\geq 0}$ a nondecreasing sequence in $\left[ \varepsilon, \frac{1-\varepsilon}{\beta}\right]$,
Algorithm  \ref{alg1} looks for this particular problem like
	\begin{align*}
	  \left(\forall n\geq 0\right) \   \left\lfloor \begin{array}{l}
		p_{1,n} = \proj_{\left[0,1\right]^n} \left(x_n - \gamma_n \left( A^*v_{1,n} + L^* v_{2,n} + \lambda_2 \mathbbm{1}^n \right) \right) \\
		p_{2,1,n} = \proj_{\left[-1,1\right]^n} \left(v_{1,n} + \gamma_n (A x_n - b)\right) \\
		p_{2,2,n} = \proj_S \left(v_{2,n} +\gamma_n L x_n \right) \\
		v_{1,n+1} = \gamma_n A( p_{1,n} - x_n) + p_{2,1,n} \\
		v_{2,n+1} = \gamma_n L( p_{1,n} - x_n) + p_{2,2,n} \\
		x_{n+1} = \gamma_n (A^*(v_{1,n}-p_{2,1,n}) + L^*(v_{2,n}-p_{2,2,n})) +p_{1,n}.		
		\end{array}
		\right.
	\end{align*}
Figure \ref{fig:cameraman} shows the original cameraman test image, which is part of the image processing toolbox in Matlab, the image obtained after multiplying it with the blur operator and adding after that normally distributed white Gaussian noise with standard deviation $10^{-3}$ and the image reconstructed by Algorithm  \ref{alg1} when taking as regularization parameters $\lambda_1=3\text{e-}3$ and $\lambda_2=2\text{e-}5$.

\subsection{\texorpdfstring{$TV$}{TV}-based image inpainting}\label{subsectionInpaint}

In the last section of the paper we show how image inpainting problems, which aim for recovering lost information, can be efficiently solved via the primal-dual methods investigated in this work. To this end, we consider the following $TV$-$l_1$ model
\begin{align}\label{ex-inpaint}
	\inf_{x\in\R^n}{\left\{ \lambda TV_{\iso}(x) + \|Kx-b\|_1 + \delta_{\left[0,1\right]^n}(x)\right\}},
\end{align}
where $\lambda \in \R_{++}$ is the regularization parameter and $TV_{\iso}:\R^n \rightarrow \R$ is the isotropic total variation functional and $K \in \R^{n \times n}$ is the diagonal matrix, where for $i=1,...,n$, $K_{i,i} = 0$, if the pixel $i$ in the noisy image $b\in \R^n$ is lost (in our case pure black) and $K_{i,i} = 1$, otherwise. The induced linear operator $K : \R^n \rightarrow \R^n$ fulfills $\|K\|=1$, while, in the light of the considerations made in the previous two subsections, we have that $TV_{\iso}(x)=\|Lx\|_{\times}$  for all $x \in \R^n$.

Thus, problem \eqref{ex-inpaint} can be formulated as
$$ \inf_{x\in\R^n}{\left\{  f(x) + g_1(Lx) + g_2(Kx)\right\}}, $$
where $f:\R^n \rightarrow \oR$, $f(x)=\delta_{\left[0,1\right]^n}$, $g_1:\Y \rightarrow \R$, $g_1(y_1,y_2)=\|(y_1,y_2)\|_{\times}$ and $g_2:\R^n \rightarrow \R$, $g_2(y)=\|y-b\|_1$. We solve it by Algorithm \ref{alg1}, the formulae for the proximal points involved in this iterative scheme been already given in Subsection \ref{subsectionDeblur}. Figure \ref{fig:fruits} shows the original fruit image, the image obtained from it after setting to pure black $80$\% randomly chosen pixels and the image reconstructed by Algorithm \ref{alg1} when taking as regularization parameter $\lambda=0.05$.

\begin{figure}[tb]	
	\centering
	\captionsetup[subfigure]{position=top}
	\subfloat[Original image]{\includegraphics*[viewport= 136 254 476 572, width=0.32\textwidth]{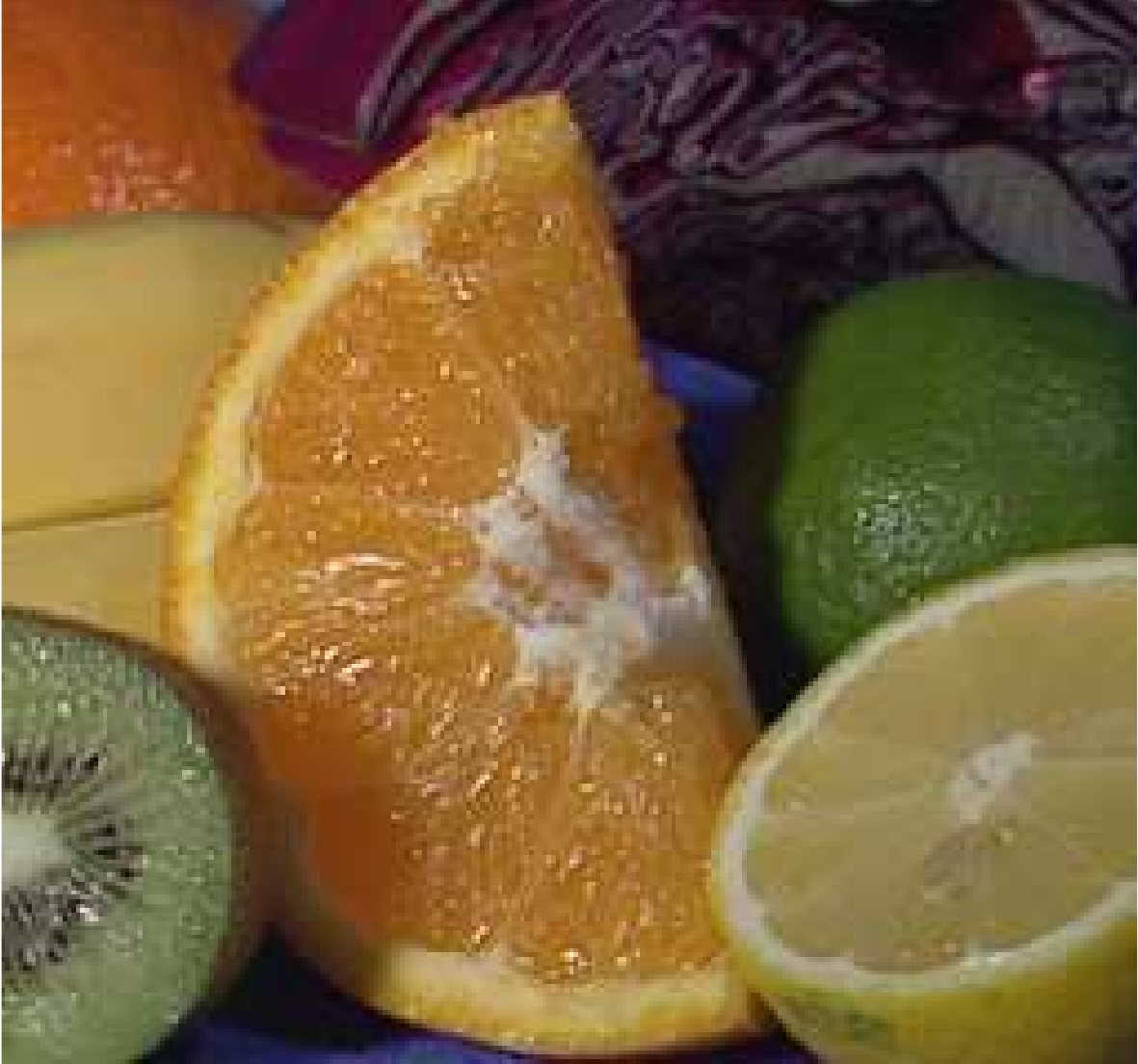}} \hspace{0.2mm}
	\subfloat[$80$\% missing pixels]{\includegraphics*[width=0.32\textwidth]{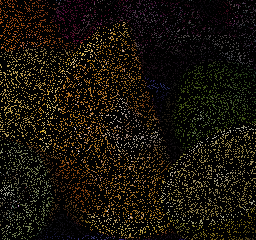}}  \hspace{0.2mm}
	\subfloat[Reconstructed image]{\includegraphics*[viewport= 136 254 476 572, width=0.32\textwidth]{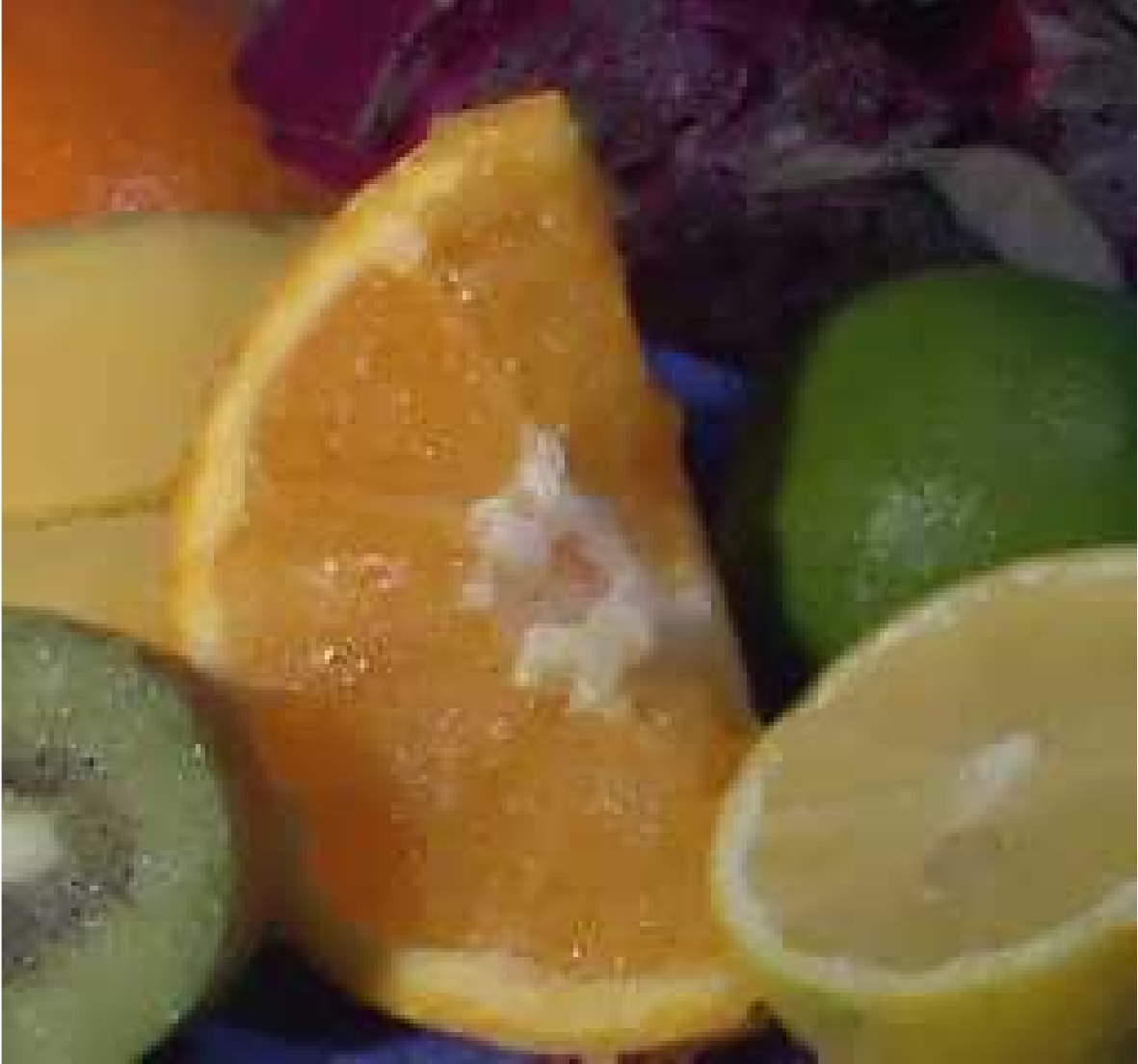}}	
	\caption{\small $TV$-$l_1$ image inpainting. Figure (a) shows the $240\times 256$ clean fruits image, (b) shows the same image for which $80$\% randomly chosen pixels were set to pure black and (c) shows the nonaveraged iterate generated by Algorithm \ref{alg1} after 200 iterations.}
	\label{fig:fruits}	
\end{figure}


\begin{thebibliography}{10}
\setlength{\itemsep}{-2pt}

\bibitem{BauschkeCombettes11}
H.H. Bauschke and P.L. Combettes.
\newblock {\em Convex Analysis and Monotone Operator Theory in Hilbert Spaces}.
\newblock CMS Books in Mathematics, Springer, 2011.

\bibitem{Bot10}
R.I. Bo\c{t}.
\newblock {\em Conjugate Duality in Convex Optimization}.
\newblock Lecture Notes in Economics and Mathematical Systems, Vol. 637, Springer-Verlag Berlin Heidelberg, 2010.

\bibitem{BotGradWanka09}
R.I. Bo{\c{t}}, S.M. Grad and G. Wanka.
\newblock {\em Duality in Vector Optimization}.
\newblock Springer-Verlag Berlin Heidelberg, 2009.

\bibitem{BriCom11}
L.M. Brice\~no-Arias and P.L. Combettes.
\newblock A monotone + skew splitting model for composite monotone inclusions in duality.
\newblock {\em 	SIAM Journal on Optimization} 21(4):1230--1250, 2011.

\bibitem{Cha04}
A. Chambolle.
\newblock An algorithm for total variation minimization and applications.
\newblock {\em Journal of Mathematical Imaging and Vision} 20(1--2):89--97, 2004.

\bibitem{ChaPoc11}
A. Chambolle and T. Pock.
\newblock A first-order primal-dual algorithm for convex problems with applications to imaging.
\newblock {\em 	Journal of Mathematical Imaging and Vision} 40(1):120--145, 2011.

\bibitem{ComPes12}
P.L. Combettes and J.-C. Pesquet.
\newblock Primal-dual splitting algorithm for solving inclusions with mixtures of composite, Lipschitzian, and parallel-sum type monotone operators.
\newblock {\em Set-Valued and Variational Analysis} 20(2):307--330, 2012.

\bibitem{Vu11}
B.C. V\~u.
\newblock A splitting algorithm for dual monotone inclusions involving cocoercive operators.
\newblock {\em Advances in Computational Mathematics}, 2011. \url{http://dx.doi.org/10.1007/s10444-011-9254-8}

\bibitem{Zalinescu02}
C. Z{\u{a}}linescu.
\newblock {\em Convex Analysis in General Vector Spaces}.
\newblock World Scientific, 2002.

\end{thebibliography}
\end{document}